\documentclass{cocv}

\usepackage{hyperref,color}
\usepackage[pdftex]{graphicx}
\usepackage{enumerate}
\usepackage{cite}

\theoremstyle{definition}

\newtheoremstyle{named}{}{}{\itshape}{}{\bfseries}{.}{.5em}{\thmnote{#3}}
\theoremstyle{named}

\newcommand{\R}{{\mathord{\mathbb R}}}

\newcommand{\Rd}{{\mathord{\mathbb R}^d}}

\newcommand{\id}{{\mathop{\rm \mathbf{id} }}}
\newcommand{\grad}{\nabla}

\newcommand{\la}{\langle}
\newcommand{\ra}{\rangle}

\def\P{{\mathcal P}}

\newcommand{\argmin}{\operatornamewithlimits{argmin}}
\newcommand{\bxi}{\boldsymbol{\xi}}
\newcommand{\bt}{\mathbf{t}}

\newcommand{\bmu}{\boldsymbol{\mu}}

\newcommand{\bh}{\boldsymbol{h}}
\newcommand{\bomega}{\boldsymbol{\omega}}
\newcommand{\bgamma}{\boldsymbol{\gamma}}

\makeatletter
\newcommand{\subjclassname@MSC}{AMS Subject Classification}
\makeatother

\begin{document}

\title{The Exponential Formula for the Wasserstein Metric}
\thanks{Work partially supported by U.S.
National Science Foundation grants DMS 0901632 and DMS 1201354.}
\author{Katy Craig} \address{Department of Mathematics, Rutgers University, 110 Frelinghuysen Road, Piscataway, NJ 08854-8019, kcraig@math.ucla.edu}

\begin{abstract}
A recurring obstacle in the study of Wasserstein gradient flow is the lack of convexity of the square Wasserstein metric. In this paper, we develop a class of \emph{transport metrics} that have better convexity properties and use these metrics to prove an Euler-Lagrange equation characterizing Wasserstein discrete gradient flow. We then apply these results to give a new proof of the exponential formula for the Wasserstein metric, mirroring Crandall and Liggett's proof of the corresponding Banach space result \cite{CrandallLiggett}. We conclude by using our approach to give simple proofs of properties of the gradient flow, including the contracting semigroup property and energy dissipation inequality.
\end{abstract}

\maketitle

\section*{Introduction} \label{section 2}

Given a continuously differentiable, convex function $E: \Rd \to \R \cup \{ +\infty \}$, the \emph{gradient flow} of $E$ is the solution to the Cauchy problem
\begin{align}
\frac{d }{dt} u(t)= - \grad E(u(t)),\quad u(0) \in \overline{D(E)}=\overline{\{v \in \Rd : E(v) < +\infty \}} \ . \label{Euclidgradflow}
\end{align}
Through suitable generalizations of the notion of the gradient, the theory of gradient flow has been extended to Hilbert spaces \cite{Brezis2}, Banach spaces \cite{Crandall, CrandallLiggett}, nonpositively curved metric spaces \cite{Mayer}, and general metric spaces \cite{AGS, ClementDesch}, including the space of probability measures endowed with the Wasserstein metric $W_2$.
 
Gradient flow in the Wasserstein metric is of particular interest due to the sharp estimates it can provide on long-time behavior of solutions to partial differential equations \cite{O,O2} and the low regularity it requires, allowing one to pass seamlessly between Lagrangian and Eulerian perspectives \cite{5person, 5person2}. 
For a sufficiently regular functional $E$ and measure $\mu$, the formal Wasserstein gradient is given by $\grad_{W_2} E(\mu) = - \grad \cdot (\mu \grad \frac{\delta E}{\delta \mu})$ and the gradient flow of $E$ corresponds to the partial differential equation 
\begin{align} \label{formal W2 grad flow}
 \frac{d}{d t} \mu(t) =  - \grad_{W_2} E(\mu)= \grad \cdot \left(\mu \grad \frac{\delta E}{\delta \mu} \right) \ ,
 \end{align}
in the duality with $C^\infty_c(\Rd \times (0,+\infty))$  \cite[Lemma 10.4.1, Theorem 11.1.4]{AGS}.

Due to the formal nature of the Wasserstein gradient, rigorous study of Wasserstein gradient flow often considers a time discretization of the problem that doesn't require a rigorous notion of gradient.
For Euclidean gradient flow (\ref{Euclidgradflow}), this discretization is simply the implicit Euler method, and it is a classical result that $u_n$ is a finite difference approximation to the gradient flow if and only if $u_n$ solves a sequence of minimization problems,
\begin{align} \label{implicit Euler}
(u_n - u_{n-1})/\tau = - \grad E(u_n) \ , u_0 = u  \quad \iff \quad u_n =\argmin_v \{ |v-u_{n-1}|^2/2\tau + E(v) \} \ , u_0 = u \ . 
\end{align}
Motivated by this, the Wasserstein \emph{proximal map} $J_\tau$ is defined by
\[ J_\tau \mu =\argmin_{\nu} \left\{W_2^2(\nu,\mu)/2\tau + E(\nu) \right\} \ , \]
and the \emph{discrete gradient flow} is given by $J^n_\tau \mu$. Taking $\tau = \frac{t}{n}$ and sending $n \to \infty$ gives the \emph{exponential formula} relating the discrete gradient flow to the gradient flow,
\begin{align*}
\lim_{n \to \infty}J_{t/n}^n \mu = \mu(t) \ .
\end{align*}
Ambrosio, Gigli, and Savar\'e were the first to prove the exponential formula for the Wasserstein metric, obtaining the sharp rate of convergence of the discrete gradient flow $J_{t/n}^n \mu$ to the gradient flow $\mu(t) $ \cite[Theorem 4.0.4]{AGS}. In the same work, they  raised the question of whether another proof might be possible, mirroring Crandall and Liggett's approach for the Banach space case\cite{CrandallLiggett}.

A recurring obstacle in the theory of Wasserstein gradient flow is that, while $x \mapsto \frac{1}{2} |x- y |^2$
is 1-convex along geodesics in Euclidean space, the square Wasserstein metric $\mu \mapsto \frac{1}{2}W_2^2(\mu, \omega)$ 
is not convex along geodesics \cite[Example 9.1.5]{AGS}. In fact, it is semi-concave \cite[Theorem 7.3.2]{AGS}. Convexity of the square Euclidean distance ensures that $v \mapsto |v-u_{n-1}|^2/2\tau + E(v)$ is convex, so the minimization problem in (\ref{implicit Euler}) is well-posed. This argument fails in the Wasserstein case. Convexity of the square distance is also essential to concluding that the proximal map satisfies the contraction inequality $ \|J_\tau u - J_\tau v\| \leq \|u -v\|$, 
a key element in Crandall and Liggett's proof of the exponential formula  \cite{Mayer, CrandallLiggett}. It is unknown if such a contraction holds in the Wasserstein case, though ``almost'' contraction inequalities exist, such as Ambrosio, Gigli, and Savar\'e's  \cite[Lemma 4.2.4]{AGS} and Carlen and the author's \cite[Theorem 1.3]{CC}.

In order to circumvent these difficulties, Ambrosio, Gigli, and Savar\'e introduced a new class of curves---\emph{generalized geodesics}---along which the square distance is 1-convex \cite[Lemma 9.2.1]{AGS}. They also introduced \emph{pseudo-Wasserstein metrics}, which they used to study the semi-concavity and differentiability of the square Wasserstein metric \cite[Equation 7.3.2]{AGS}. In this paper, we further develop these ideas, introducing a class of \emph{transport metrics}, which are a type of pseudo-Wasserstein metric. We show that generalized geodesics, while not geodesics with respect to the Wasserstein metric, are  geodesics with respect to the transport metrics. The transport metrics also satisfy the key property that the square distance is 1-convex. We use the transport metrics to prove an Euler-Lagrange equation characterizing the discrete gradient flow, analogous to equation (\ref{implicit Euler}) above. (One direction of this characterization is due to Ambrosio, Gigli, and Savar\'e \cite{AGS}. We prove the converse.) We then apply the transport metrics and Euler-Lagrange equation to give a new proof of the exponential formula, in the style of Crandall and Liggett. We conclude by applying our estimates to give simple proofs of properties of the gradient flow, including the contracting semigroup property and energy dissipation inequality.

We are optimistic that our results will be useful in future study of Wasserstein gradient flow. Our Euler-Lagrange equation replaces the global minimization problem in the definition of the proximal map with a local computation of the subdifferential. Our results on the structure of transport metrics give further credence to the geometric naturalness of the assumption of convexity along generalized geodesics, a key element in the work of Ambrosio, Gigli, and Savar\'e. Our new proof of the exponential formula brings together the theories of Banach space and Wasserstein gradient flow, and we believe our work may be useful in studying the behavior of the gradient flow as the functional $E$ varies---for example, as a regularization of $E$ is removed. The corresponding problem in the Banach space case is well-understood \cite{BrezisPazy}, and the analogy we establish between the Banach and Wasserstein cases may help extend these results.

\section{Transport Metrics and the Euler-Lagrange Equation} \label{first section}
In sections \ref{Wasserstein Metric subsection}-\ref{discrete grad flow section} we recall foundational results on Wasserstein gradient flow, including Ambrosio, Gigli, and Savar\'e's notion of generalized geodesics and pseudo-Wasserstein metrics. In section \ref{transport metric section} we build on this theory, developing the structure of transport metrics, which have better convexity properties than the Wasserstein metric. In section \ref{el section}, we use the convexity of the square transport metrics to prove an Euler-Lagrange equation characterizing the discrete gradient flow.

\subsection{Wasserstein Metric} \label{Wasserstein Metric subsection}
We begin by recalling properties of the Wasserstein metric. Let $\P(\Rd)$ denote the set of probability measures on $\Rd$. Given $\mu, \nu \in \P(\Rd)$, a measurable function $\bt: \Rd \to \Rd$ \emph{transports $\mu$ onto $\nu$} if $\nu(B) = \mu(\bt^{-1}(B))$ for all measureable $B \subseteq \Rd$. We call $\nu$ the \emph{push-forward of $\mu$ under $\bt$} and write $\nu = \bt \# \mu$.

For a finite product $\Pi_{i=1}^n \Rd$, let $\pi^i$ the be projection onto the $i$th component and $\pi^{i,j}$ be the projection onto the $i$th and $j$th components. 
Given $\mu, \nu \in P(\Rd)$, the set of \emph{transport plans} from $\mu$ to $\nu$ is 
\[ \Gamma(\mu, \nu)= {\{ \bmu \in \P(\Rd \times \Rd) : {\pi^1 \# \bmu = \mu}, {\pi^2 \# \bmu = \nu} \}} \ . \]
The \emph{Wasserstein distance} between $\mu$ and $\nu$ is defined to be
\begin{align}
 W_2(\mu,\nu) =  \inf \left\{  \left(\int_{\Rd \times \Rd} |x_1-x_2|^{2}d \bmu(x_1,x_2)  \right)^{1/2}:  \bmu \in \Gamma(\mu, \nu) \right\}. \label{W2}
 \end{align}
If $W_2(\mu,\nu)< +\infty$, the infimum is attained, and we denote the set of \emph{optimal transport plans} by $\Gamma_0(\mu, \nu)$.

If $\mu$ does not charge sets of $d-1$ dimensional Hausdorff measure, we say that $\mu$ \emph{does not charge small sets}. In this case, there is a unique optimal transport plan from $\mu$ to $\nu$ of the form $(\id \times \bt) \# \mu$, where $\bt$ is unique $\mu$-a.e. and $\id(x) = x$ is the identity transformation\cite{McCannExistence}. In particular, there is an \emph{optimal transport map} $\bt = \bt_\mu^\nu$ satisfying $\bt \# \mu = \nu$ and $W_2(\mu,\nu) = \| \id - \bt \|_{L^2(\mu)}$. (See Gigli \cite{Gigli} for a sharp version of this result.) Furthermore, a measurable map satisfying $\bt \# \mu = \nu$ is optimal if and only if it is cyclically monotone $\mu$-a.e. \cite{McCannExistence}.
If $\nu$ also does not charge small sets, then $\bt_\mu^\nu \circ \bt_\nu^\mu = \id$ almost everywhere with respect to $\mu$.

One technical difficulty when working with the Wasserstein distance on $\P(\Rd)$ is that there exist measures that are infinite distances apart. 
 Given a fixed reference measure $\omega_0$, define 
\[ \P_{2, \omega_0}(\Rd) = \{ \mu \in \P(\Rd) : W_2(\mu,\omega_0)< +\infty \} \ .\]
By the triangle inequality, $(\P_{2, \omega_0}(\Rd), W_2)$ is a metric space. When $\omega_0 = \delta_0$, the Dirac mass at the origin, $\P_{2, \omega_0}(\Rd)=\P_2(\Rd)$, the subset of $\P(\Rd)$ with finite second moment.

We consider three classes of curves: locally absolutely continuous curves, (constant speed) geodesics, and generalized geodesics. We define the first two now and leave the third for the next section.
\begin{defi} \label{abs cts def}
$\mu(t): \R \to \P_{2, \omega_0}(\Rd)$ is \emph{locally absolutely continuous} if for all $I \subseteq \R$ bounded, there exists $m \in L^1(I)$ so that $W_2(\mu(t),\mu(s)) \leq \int_s^t m(r) dr$ for all $s \leq t \in I$.
\end{defi}
\begin{defi} \label{geodef}
Given a metric space $(X,d)$, $u_\alpha: [0,1] \to X$ is a (constant speed) \emph{geodesic} in case $d(u_\alpha,u_\beta) = |\beta-\alpha| d(u_0,u_1)$ for all $\alpha, \beta \in [0,1]$.
\end{defi}
As any two measures are connected by a geodesic,  $(\P_{2, \omega_0}(\Rd), W_2)$ is a geodesic metric space, and all geodesics are of the form $\mu_\alpha = \left( (1-\alpha)\pi^1 + \alpha \pi^2 \right) \# \bmu$ for $\bmu \in \Gamma_0(\mu_0,\mu_1)$ \cite[Theorem 7.2.2]{AGS}.
If $\mu_0$ does not charge small sets, the geodesic from $\mu_0$ to $\mu_1 $ is unique and given by $\mu_\alpha = \left( (1-\alpha)\id + \alpha \mathbf{t}_{\mu_0}^{\mu_1} \right) \# \mu_0$.

Given a metric space $(X, d)$, we place the following conditions on functionals $E: X \to \R \cup \{+\infty\}$:

\begin{itemize}
\item \textit{proper:} $ D(E) = \{u \in X : E(u) < +\infty \} \neq \emptyset$

\item \textit{coercive}: 
There exists $\tau_0 > 0$, $u_0 \in X$ so that $
 \inf \left\{\frac{1}{2 \tau_0} d^2(u_0,v) + E(v) : v \in X \right\}> - \infty.$

\item \textit{lower semicontinuous:} For all $u_n, u \in X$ so that $u_n \to u$, $\liminf_{n \to \infty} E(u_n) \geq E(u)$.

\item \textit{$\lambda$-convex along a curve $u_\alpha$:} Given $\lambda \in \R$ and a curve $u_\alpha \in X$,
\begin{align}
E(u_\alpha) \leq (1-\alpha)E(u_0) + \alpha E(u_1) -  \alpha (1- \alpha)\frac{\lambda}{2}d(u_0,u_1)^2 \ . \label{convexdef}
\end{align}

\item \textit{$\lambda$-convex along geodesics}: Given $\lambda \in \R$, for all $u_0, u_1 \in X$, there exists a geodesic $u_\alpha$ from $u_0$ and $u_1$ along which (\ref{convexdef}) holds. We will often simply say that $E$ is \emph{$\lambda$-convex}, or in the case $\lambda=0$, \emph{convex}.
\end{itemize}

If $(X, d) = (\P_{2, \omega_0}(\Rd), W_2)$, we consider the following notions of differentiability.
\begin{defi} \label{metric slope def}
Given $E: \P_{2, \omega_0}(\Rd) \to \R \cup \{+\infty\}$,  the \emph{metric slope} of $E$ at $u \in D(E)$ is
\[ |\partial E|(u) = \limsup_{v \to u} \frac{(E(u) - E(v))^+}{d(u,v)} \ . \]
\end{defi}
For ease of notation, we suppose $E$ satisfies the following assumption, which ensures that for all $\mu \in D(E)$ and $\nu \in \P(\Rd)$ there exists an optimal transport map $\bt_\mu^\nu$ from $\mu$ to $\nu$. 
\begin{as}\label{small set assumption}
If $\mu \in D(E)$, $\mu$ does not give mass to small sets.
\end{as}
\begin{defi} \label{Wsubdiff}
Suppose $E: \P_{2, \omega_0}(\Rd) \to \R \cup \{+\infty\}$ satisfies domain assumption \ref{small set assumption} and is proper, coercive, lower semicontinuous, and $\lambda$-convex along geodesics. Then $\bxi \in L^2(\mu)$ belongs to the \emph{subdifferential} of $E$ at $\mu$ in case for all $\nu \in D(E)$,
\[ E(\nu) - E(\mu) \geq \int_\Rd \la \bxi , \bt_\mu^\nu - \id \ra d \mu + \frac{\lambda}{2}W_2(\mu,\nu) \ . \]
We denote this by $\bxi \in \partial E (\mu)$. In addition, $\bxi$ is a \emph{strong subdifferential} if for all measurable $\bt$ it satisfies
\[ E(\mathbf{t} \# \mu) - E(\mu) \geq \int_{\Rd} \la \bxi , \mathbf{t}-\id
\ra d \mu + o(\|\mathbf{t} - \id\|_{L^2(\mu)})  \ \quad \text{ as } t \xrightarrow{L^2(\mu)} \id \ . \]
\end{defi}

\subsection{Generalized Geodesics}
A recurring difficulty in extending results from a Hilbert space $(\mathcal{H}, \|\cdot \|)$ to the Wasserstein metric $(\P_{2, \omega_0}, W_2)$ is that while $y \mapsto \frac{1}{2} \|x- y \|^2$ is 1-convex along geodesics, $\mu \mapsto \frac{1}{2}W_2^2(\omega, \mu)$
is not \cite[Example 9.1.5]{AGS}.
Ambrosio, Gigli, and Savar\'e circumvent this difficulty by introducing generalized geodesics \cite[Definition 9.2.2]{AGS}.

\begin{defi}
Given $\mu_0, \mu_1, \omega \in \P_{2, \omega_0}(\Rd)$, a measure $\bomega \in \P(\Rd \times \Rd \times \Rd)$ is a \emph{transport plan from $\mu_0$ to $\mu_1$ with base $\omega$} in case $\pi^{1,2} \# \bomega \in \Gamma_0(\omega,\mu_0) \text{ and }\pi^{1,3} \# \bomega \in \Gamma_0(\omega,\mu_1)$.
\end{defi}

\begin{defi} \label{gen geo def}
Given $\mu_0, \mu_1, \omega \in \P_{2, \omega_0}(\Rd)$ and $\bomega \in \P(\Rd \times \Rd \times \Rd)$ a transport plan from $\mu_0$ to $\mu_1$ with base $\omega$, a \emph{generalized geodesic from $\mu_0$ to $\mu_1$ with base $\omega$} is the curve 
\[\mu_\alpha = \left( (1-\alpha) \pi^2 + \alpha \pi^3 \right) \# \bomega  \ , \quad \alpha \in [0,1]. \]
\end{defi}
For any three measures, a transport plan from $\mu_0$ to $\mu_1$ with base $\omega$ always exists always exists \cite[Lemma 5.3.2]{AGS}, hence generalized geodesics always exist. If $\omega$ is absolutely continuous with respect to Lebesgue measure, the generalized geodesic is unique and defined by $\mu_\alpha= \left( (1-\alpha) \mathbf{t}_{\omega}^{\mu_0} + \alpha \mathbf{t}_{\omega}^{\mu_1} \right) \# \omega$. Typically, a generalized geodesic is \emph{not} a geodesic. However, it is if the base $\omega$ coincides with either $\mu_0$ or $\mu_1$.

In addition to the notion of generalized geodesics, Ambrosio, Gigli, and Savar\'e introduced the related notion of pseudo-Wasserstein metrics \cite[Equation 9.2.5]{AGS}.
\begin{defi}\label{pseudo metric def}
Given a transport plan $\bomega$ from $\mu_0$ to $\mu_1$ with base $\omega$, the \emph{pseudo-Wasserstein metric }is 
\[ W_{2,\bomega}(\mu_0,\mu_1) = \left(\int_{\Rd \times \Rd} |x_2-x_3|^{2}d \bomega(x_1,x_2,x_3)  \right)^{1/2} \ . \]
\end{defi}
\begin{remark} \label{transport metric bounds W2}
If $\omega = \mu_0$ or $\mu_1$, this reduces to the Wasserstein metric. In general, ${W_{2,\bomega}(\mu_0,\mu_1) \geq W_2(\mu_0,\mu_1)}$. We also have $W_{2,\bomega}(\mu_0,\mu_1) \leq W_2(\mu_0,\omega) + W_2(\omega, \mu_1)$ by the triangle inequality for $L^2(\bomega)$.
\end{remark}

Let $\bomega$ be a transport plan from $\mu_0$ to $\mu_1$ with base $\omega$. If $\mu_\alpha$ is the generalized geodesic induced by $\bomega$ and $W_{2,\bomega}$ is the corresponding pseudo-Wasserstein metric, Ambrosio, Gigli, sand Savar\'e showed
\begin{align} \label{pseudo metric Hilbertian identity}
W_2^2(\omega, \mu_\alpha) = (1-\alpha) W_2^2(\omega,\mu_0) + \alpha
W_2^2(\omega,\mu_1) - \alpha(1-\alpha)W_{2,\bomega}^2(\mu_0,\mu_1) \ , \quad \forall \alpha \in [0,1] \ .
\end{align}
In particular, while $\mu \mapsto \frac{1}{2}W_2^2( \omega,\mu)$ is not convex along geodesics, it is 1-convex along \emph{generalized geodesics} with base $\omega$ \cite[Lemma 9.2.1]{AGS}. This convexity is a key element in their study of discrete gradient flow.

\subsection{Gradient Flow and Discrete Gradient Flow} \label{discrete grad flow section}
For $E: \P_{2, \omega_0}(\Rd) \to \R \cup \{ +\infty \}$ and $\tau >0$, the \emph{quadratic perturbation} $\Phi$ and \emph{proximal map} $J_\tau$ are defined by 
\begin{align} \label{quadper}
\Phi(\tau, \mu; \nu)= \frac{1}{2\tau} W_2^2(\mu,\nu) + E(\nu) \ , \quad J_\tau(\mu)= \argmin_{\nu \in \P_{2, \omega_0}(\Rd)}  \Phi(\tau, \mu;\nu) \ . 
\end{align}
Let $J_0(\mu) = \mu$. The \emph{discrete gradient flow} sequence of $E$ is constructed via repeated applications of $J_\tau$,
\[ \mu_n = J_\tau (\mu_{n-1}) = J^n_\tau (\mu_0) \ , \quad \mu_0 \in \overline{D(E)} \ . \]
In order to ensure the proximal map is well-defined, Ambrosio, Gigli, and Savar\'e require that the quadratic perturbation $\Phi(\tau, \mu; \cdot)$ is \emph{convex along generalized geodesics} with base $\mu$ \cite[Definition 9.2.2]{AGS}.

\begin{defi}\label{convexgengeodefi}
Given $\lambda \in \R$, a functional $E: \P_{2, \omega_0}(\Rd) \to \R \cup \{ +\infty \}$ is \emph{$\lambda$-convex along a generalized geodesic $\mu_\alpha$} induced by $\bomega$, in case
\begin{align}
E(\mu_\alpha) \leq (1-\alpha)E(\mu_0) + \alpha E(\mu_1) -  \alpha (1- \alpha)\frac{\lambda}{2}W^2_{2,\bomega}(\mu_0,\mu_1) \ , \quad \forall \alpha \in [0,1]\ . \label{convexgengeodef}
\end{align}
$E$ is \emph{convex along generalized geodesics} if, for all $\mu_0, \mu_1, \omega \in \P_{2, \omega_0}(\Rd)$, there exists a generalized geodesic from $\mu_0$ to $\mu_1$ with base $\omega$ along which $E$ is convex.
\end{defi}

By equation (\ref{pseudo metric Hilbertian identity}), to ensure $\Phi(\tau, \mu; \cdot)$ is convex along generalized geodesics with base $\mu$, it is enough for $E$ to be $\lambda$-convex along generalized geodesics for $0< \tau< \frac{1}{\lambda^{-}}$ (where $\lambda^- = \max \{0, -\lambda\}$).
 Note that if $E$ is $\lambda$-convex along generalized geodesics, it is also $\lambda$-convex along geodesics. Going forward, we often assume the following:
 \begin{as} \label{funas}
$E$ is proper, coercive, lower semicontinuous, and $\lambda$-convex along generalized geodesics.
\end{as}
\noindent With this, the proximal map $J_\tau: \overline{D(E)} \to D(E): \mu \mapsto \mu_\tau$ is well-defined and continuous \cite[Theorem 4.1.2]{AGS}.

We now define the Wasserstein gradient flow of a functional $E$.
\begin{defi} \label{gradflowdef} Suppose $E: \P_{2, \omega_0}(\Rd) \to \R \cup \{ +\infty \}$ is proper, coercive, lower semicontinuous, and $\lambda$-convex along generalized geodesics for $\lambda \in \R$.
A locally absolutely continuous curve $\mu: (0,+\infty) \to \P_{2,\omega_0}(\Rd)$ is the\emph{ gradient flow} of $E$ with initial data $\mu \in \overline{D(E)}$ if $\mu(t) \xrightarrow{t \to 0} \mu$ and
\begin{align} \label{continuous evi}
\frac{1}{2} \frac{d}{dt} W_2^2(\mu(t), \omega) + \frac{\lambda}{2} W_2^2( \mu(t), \omega) \leq E(\omega) - E( \mu(t))  ,  \quad \forall \omega \in D(E) , \text{ a.e. }  t >0 \ .
\end{align}
\end{defi}

We close this section by recalling two inequalities for the discrete gradient flow that are consequences of the convexity of $\Phi(\tau, \mu; \cdot)$ along generalized geodesics with base $\mu$ \cite[Theorems 3.1.6 and 4.1.2]{AGS}. Suppose $E$ is proper, coercive, lower semicontinuous, and $\lambda$-convex along generalized geodesics with ${0 < \tau< 1/\lambda^-}$. Then for $\mu \in D(|\partial E|)$,
\begin{align} \label{theorem AGS1}
 \tau^2 | \partial E|^2(J_\tau \mu) &\leq  W^2_2(\mu,J_\tau \mu) \leq \frac{2\tau}{1+\lambda \tau} \left[ E(\mu) - E(J_\tau \mu) - \frac{1}{2 \tau} W_2^2(\mu,J_\tau \mu)\right] \leq \frac{\tau^2}{(1+\lambda \tau)^2} | \partial E|^2(\mu) \ .
\end{align}
For $\mu \in \overline{D(E)}$ and $\nu \in D(E)$,
\begin{align} \label{discevi}
\frac{1}{2\tau} [W_2^2(J_\tau \mu, \nu)-W_2^2(\mu,\nu) ] + \frac{\lambda}{2} W_2^2(J_\tau \mu, \nu)\leq E(\nu)- E(J_\tau \mu) -
\frac{1}{2\tau} W_2^2(\mu,J_\tau\mu) \ . 
\end{align}

\subsection{Transport Metrics} \label{transport metric section}
We now consider further properties of the pseudo-Wasserstein metric in the particular case that the base measure does not give mass to small sets. In contrast to the previous sections, in which we reviewed existing results, the results in the current section are new and play a key role in our proof of the Euler-Lagrange equation and our proof of the exponential formula.

First, we show the following generalization of (\ref{pseudo metric Hilbertian identity}).
\begin{prop} \label{convexity of pseudo metrics}
Fix $\mu_0, \mu_1, \omega \in \P_{2,\omega_0}(\Rd)$. Let $\bomega$ be a transport plan from $\mu_0$ to $\mu_1$ with base $\omega$ and let $\mu_\alpha$ be the generalized geodesic induced by $\bomega$. Then, for all $\nu \in P_{2,\omega_0}(\Rd) $ there exists a transport plan $\bomega_\alpha$ from $\nu$ to $\mu_\alpha$ with base $\omega$ so that
\[W_{2,\bomega_\alpha}^2(\nu, \mu_\alpha) = (1-\alpha) W_{2,\bomega_0}^2(\nu,\mu_0) + \alpha W_{2,\bomega_1}^2(\nu,\mu_1) - \alpha(1-\alpha)W_{2,\bomega}^2(\mu_0,\mu_1) \ , \quad \forall \alpha \in [0,1] \ .\]
\end{prop}

\begin{proof}
Fix $\bmu \in \P(\Rd \times \Rd \times \Rd \times \Rd)$ so that $\pi^{1,2} \# \bmu = \pi^{1,2} \# \bomega \in \Gamma_0(\omega, \mu_0),$ $\pi^{1,3} \# \bmu = \pi^{1,3} \# \bomega \in \Gamma_0(\omega,\mu_1)$, and $\pi^{1,4} \# \bmu \in \Gamma_0(\omega, \nu)$ \cite[Lemma 5.3.4]{AGS}. Define $\bomega_\alpha = (\pi^1 \times \pi^4 \times [(1-\alpha) \pi^2 + \alpha \pi^3]) \# \bmu$. Then $\bomega_\alpha$ is a transport plan from $\nu$ to $\mu_\alpha$ with base $\omega$ and, by the corresponding identity for $L^2(\bmu)$,
\[ W_{2,\bomega_\alpha}^2(\nu, \mu_\alpha) = \|x_4 - (1-\alpha)x_2 + \alpha x_3 \|^2_{L^2(\bmu)} =  (1-\alpha) W_{2,\bomega_0}^2(\nu,\mu_0) + \alpha W_{2,\bomega_1}^2(\nu,\mu_1) - \alpha(1-\alpha)W_{2,\bomega}^2(\mu_0,\mu_1) \ . \]
\end{proof}

Now we consider the pseudo-Wasserstein metrics in the particular case that the base $\omega$ doesn't charge small sets.
In this case, it becomes a true metric, and to emphasize this point, we call it the $(2,\omega)$-transport metric.
\begin{defi} \label{transport distance def}
The \emph{$(2, \omega)$-transport metric} $W_{2, \omega}: \P_{2,\omega}(\Rd) \times \P_{2,\omega}(\Rd) \to \R$ is given by
\begin{align*} 
 W_{2,\omega}&(\mu_0,\mu_1) = \left( \int |\mathbf{t}_\omega^{\mu_0} -
\mathbf{t}_\omega^{\mu_1} |^2 d \omega \right)^{1/2} \ .
\end{align*}
\end{defi}
We show that the generalized geodesics with base $\omega$ are exactly the constant speed geodesics for this metric. This allows us to consider functionals which are convex with respect to the transport metric and define a notion of subdifferential with respect to this metric. The map $\mu_0 \mapsto \bt_\omega^{\mu_0}$ is a geodesic preserving isometry from $\P_{2,\omega}(\Rd) $ to $L^2(\omega)$, so the square transport metric is $1$-convex with respect to its own geodesic structure. (The 1-convexity of the square transport metric can also be seen as a special case of Proposition \ref{convexity of pseudo metrics}.)
\begin{prop} \label{rhoWprop}
\quad
\begin{enumerate}[(i)]
\item $W_{2,\omega}$ is a metric on $\P_{2,\omega}(\Rd)$.
\item The constant speed
geodesics with respect to $W_{2,\omega}$ are the generalized geodesics with base $\omega$.
\item $\mu \mapsto \bt_\omega^{\mu}$ is a geodesic preserving isometry from $\P_{2,\omega}(\Rd) $ to $L^2(\omega)$, hence for all $\nu \in P_{2,\omega}(\Rd)$,
\begin{align} \label{convalonggengeo}
W_{2,\omega}^2(\nu, \mu_\alpha) = (1-\alpha) W_{2,\omega}^2(\nu,\mu_0) + \alpha
W_{2,\omega}^2(\nu,\mu_1) - \alpha(1-\alpha)W_{2,\omega}^2(\mu_0,\mu_1) \quad \forall \alpha \in [0,1] \ .
\end{align}
\end{enumerate}
\end{prop}

\begin{proof}
First, we show (i). $W_{2,\omega}$ is symmetric and nonnegative by definition. It is
non-degenerate since $W_{2,\omega}(\mu,\nu) \geq W_2(\mu,\nu)$. It satisfies the triangle inequality since $L^2(\omega)$ satisfies the
triangle inequality.

Next, we show that generalized geodesics with base $\omega$ are constant speed geodesics in $W_{2, \omega}$.
 Let $\mu_\alpha = ((1-\alpha)\bt_\omega^{\mu_0} + \alpha
\bt_\omega^{\mu_1})\# \omega$ be the generalized geodesic with base $\omega$ from
$\mu_0$ to $\mu_1$. Since $(1-\alpha)\bt_\omega^{\mu_0} + \alpha \bt_\omega^{\mu_1}$ is a convex combination of cyclically monotone maps, $\bt_\omega^{\mu_\alpha} = (1-\alpha)\bt_\omega^{\mu_0} + \alpha \bt_\omega^{\mu_1}$, hence
\begin{align*}
W_{2,\omega}(\mu^{\mu \to \nu}_\alpha,\mu^{\mu \to \nu}_\beta) &= \| [(1-\alpha)\bt_\omega^\mu + \alpha \bt_\omega^\nu] - [(1-\beta)\bt_\omega^\mu +
\beta \bt_\omega^\nu] \|_{L^2(\omega)}= |\beta - \alpha |W_{2,\omega}(\mu,\nu) \ .
\end{align*}

Since the isometry $\mu \mapsto \bt_\omega^\mu$ sends $\mu_\alpha \mapsto t_\omega^{\mu_\alpha} =  (1-\alpha)\bt_\omega^{\mu_0} + \alpha \bt_\omega^{\mu_1}$, equation (\ref{convalonggengeo}) holds for $\mu_\alpha$ by the the corresponding identity for $L^2(\omega)$. 
It remains to show the $\mu_\alpha$ is the unique geodesic from  $\mu_0$ to $\mu_1$. Suppose $\tilde{\mu}_\alpha$ is another. Setting $\nu = \tilde{\mu}_\alpha$ in equation (\ref{convalonggengeo}) gives
\begin{align*}
W_{2,\omega}^2(\tilde{\mu}_\alpha, \mu_\alpha) &= (1-\alpha) W_{2,\omega}^2(\tilde{\mu}_\alpha,\mu_0) + \alpha
W_{2,\omega}^2(\tilde{\mu}_\alpha,\mu_1) - \alpha(1-\alpha)W_{2,\omega}^2(\mu_0,\mu_1) \\
&= (1-\alpha) \alpha^2W_{2,\omega}^2(\mu_0,\mu_1) + \alpha (1-\alpha)^2
W_{2,\omega}^2(\mu_0,\mu_1) - \alpha(1-\alpha)W_{2,\omega}^2(\mu_0,\mu_1) =0 \ .
\end{align*}
\end{proof}

\begin{remark}[convexity] \label{convexity remark}
By (ii), if $E$ is $\lambda$-convex along generalized geodesics with base $\omega$, it is $\lambda$-convex in the $W_{2,\omega}$
metric. By (iii), the function $\mu \mapsto W_{2,\omega}^2(\nu,\mu)$ is
2-convex in the $W_{2,\omega}$ metric for any $\nu \in \P_{2, \omega}(\Rd)$. Note the difference between (\ref{pseudo metric Hilbertian identity}), which ensures $\mu \mapsto W^2_2(\omega, \mu)$ is convex along generalized geodesics with base $\omega$, and (\ref{convalonggengeo}), which ensures $\mu \mapsto W^2_{2,\omega}(\nu,\mu)$ is convex along generalized geodesics with base $\omega$ for all $\nu \in P_{2,\omega}(\Rd)$.
\end{remark}
\begin{remark}[lower semicontinuity]
By Remark \ref{transport metric bounds W2}, if $\mu_n$ converges to $\mu$ in $W_{2,\omega}$, then the sequence converges in $W_2$. Therefore, if $E$ is lower semicontinuous in $W_2$, $E$ is lower semicontinuous in $W_{2, \omega}$.
\end{remark}

Using the isometry from $\P_{2,\omega}(\Rd) $ to $L^2(\omega)$, we define the $W_{2,\omega}$ subdifferential.

\begin{defi} \label{rhosubdiff}
Given $E: \P_{2, \omega}(\Rd) \to \R \cup \{+\infty\}$ proper, lower semicontinuous, and $\lambda$-convex with respect to $W_{2,\omega}$, $\bxi \in L^2(\omega)$ belongs to the
 \emph{$W_{2,\omega}$ subdifferential} of $E$ at $\mu$ in case for all $\nu \in \P_{2, \omega}(\Rd)$,
\begin{align}
E(\nu) - E(\mu) \geq \int \la \bxi, \mathbf{t}_\omega^\nu -
\mathbf{t}_\omega^\mu \ra d \omega + \frac{\lambda}{2}W_{2,\omega}^2(\mu,\nu) \ . \label{convexrhosubdiff}
\end{align}
We denote this by $\bxi \in \partial_{2,\omega} E(\mu)$.
\end{defi}

\begin{remark}[characterization of minimizers] \label{minimizersubdiffcor}
Given $E$ as in Definition \ref{rhosubdiff} with $\lambda \geq 0$, $\mu$ is a minimizer of $E$ if and only if $0 \in \partial_{2,\omega} E(\mu)$.
\end{remark}

By equation (\ref{pseudo metric Hilbertian identity}) and Remark \ref{convexity remark},  the square Wasserstein distance from $\omega$, $\mu \mapsto W_{2}^2(\omega,\mu)$, is
2-convex in $W_{2,\omega}$. Thus, we may compute its subdifferential with respect to this metric.

\begin{prop}
\label{W2rhosubdiffprop}
Let $F(\mu) = W^2_2(\omega, \mu)$. Then $2(\mathbf{t}_\omega^\mu - \id) \in \partial_{2,\omega} F(\mu)$.
\end{prop}
\begin{proof}
$W_2^2(\omega,\nu) - W_2^2(\omega, \mu) = \int|\mathbf{t}_\omega^\nu - \id|^2 d
\omega - \int | \mathbf{t}_\omega^\mu - \id|^2 d \omega =  \int 2 \la \mathbf{t}_\omega^\mu - \id, \mathbf{t}_\omega^\nu -
\mathbf{t}_\omega^\mu \ra d \omega + W_{2,\omega}^2(\mu,\nu) $.
\end{proof}

Finally, we relate the transport metric subdifferential to the strong subdifferential from Definition \ref{Wsubdiff}.

\begin{lem}
\label{subdifflemma}
Suppose $E: \P_{2, \omega}(\Rd) \to \R \cup \{+\infty\}$ satisfies domain assumption \ref{small set assumption} and is proper, lower semicontinuous, and $\lambda$-convex along generalized geodesics. Then if $\bxi \in \partial E(\mu)$ is a strong subdifferential, we have $\bxi
\circ \mathbf{t}_\omega^\mu \in \partial_{2,\omega} E(\mu)$.
\end{lem}
\begin{proof} If $\bxi \in \partial E(\mu)$, then $\bxi \in L^2(\mu)$, hence $\bxi \circ \bt_\omega^\mu \in L^2(\omega)$. Furthermore, for all $\nu \in \P_{2, \omega}(\Rd)$,
\begin{align*}
E(\nu) - E(\mu) &\geq \int_{\Rd} \la \bxi , \mathbf{t_\omega^\nu} \circ
\mathbf{t}_\mu^\omega-\id \ra d \mu + o(\|\mathbf{t_\omega^\nu} \circ
\mathbf{t}_\mu^\omega - \id\|_{L^2(\mu)}) = \int_{\Rd} \la \bxi \circ \mathbf{t}_\omega^\mu, \mathbf{t}_\omega^\nu -
\mathbf{t}_\omega^\mu \ra d \omega + o(W_{2,\omega}(\mu,\nu)) \ .
\end{align*}
As in \cite[Section 10.1.1, B]{AGS}, this implies $E(\nu) - E(\mu) \geq \int \la \bxi \circ \mathbf{t}_\omega^\mu, \mathbf{t}_\omega^\nu -
\mathbf{t}_\omega^\mu \ra d \omega + \frac{\lambda}{2}W_{2,\omega}(\mu,\nu)$.
\end{proof}

We close this section with an analogue of inequality (\ref{discevi}) for transport metrics. Note that since $W_2^2(J_\tau \mu, \nu) \leq W_{2, \mu}^2(J_\tau \mu, \nu)$, this is stronger than (\ref{discevi}). We require this strength in our proof of Theorem \ref{recineqthm}.

\begin{thm} \label{sdiscevi} Suppose $E$ is proper, coercive, lower semicontinuous, and $\lambda$-convex along generalized geodesics for $\lambda \in \R$. Then for all $\mu \in \overline{D(E)}$ and $\nu \in D(E)$,
\begin{align*}
\frac{1}{2\tau} [W_{2,\mu}^2(J_\tau \mu, \nu)-W_2^2(\mu,\nu) ] + \frac{\lambda}{2} W_{2,\mu}^2(J_\tau \mu,\nu) \leq E(\nu)- E(J_\tau \mu) -
\frac{1}{2\tau} W_2^2(\mu,J_\tau \mu)  \ .
\end{align*}
\end{thm}

\begin{proof}
If $\mu$ does not charge small sets, so $W_{2, \mu}$ is a well-defined metric, this is simply the Talagrand inequality for $\Phi(\tau, \mu;\cdot)$ in the $W_{2, \mu}$ metric. Otherwise, since both $E$ and $\frac{1}{2 \tau} W_2^2(\mu,\cdot)$ are convex along generalized geodesics with base $\mu$, so is $\Phi(\tau, \mu;\cdot)$. Thus, for any generalized geodesic $\mu_\alpha$ from $J_\tau \mu$ to $\nu$ with base $\mu$, since $J_\tau \mu$ is the minimizer of $\Phi(\tau,\mu;\cdot)$,
\[ \Phi(\tau,\mu;J_\tau \mu) \leq \Phi(\tau, \mu;\mu_\alpha) \leq (1-\alpha) \Phi(\tau,\mu; J_\tau \mu) + \alpha \Phi(\tau, \mu;\nu)-\frac{1+ \lambda \tau}{2\tau} \alpha (1-\alpha) W^2_{2,\mu}(J_\tau \mu,\nu) \ . \]
Rearranging and dividing by $\alpha$ gives $0 \leq \Phi(\tau, \mu;\nu)- \Phi(\tau,\mu; J_\tau \mu)  -\frac{1+\lambda \tau}{2\tau} (1-\alpha) W^2_{2,\mu}(J_\tau \mu,\nu)$.
Sending $\alpha \to 0$ and expanding $\Phi$ according to its definition gives the result.
\end{proof}

\subsection{Euler-Lagrange Equation} \label{el section}
In this section, we use our results on transport metrics to prove an Euler-Lagrange equation characterizing $J_\tau \mu$. The fact that $J_\tau \mu$ satisfies $\frac{1}{\tau}(\mathbf{t}_{J_\tau \mu}^\mu - \id) \in \partial E (J_\tau \mu)$
was proved by Ambrosio, Gigli, and Savar\'e  using a type of argument introduced by Otto \cite[Lemma 10.1.2]{AGS}\cite{O,O2}. The converse is new.

\begin{thm}[Euler-Lagrange equation] \label{elthm} Assume that $E$ satisfies domain assumption \ref{small set assumption} and convexity assumption \ref{funas} for $\lambda \in \R$. Then for $\mu \in D(E)$ and
 $0< \tau< \frac{1}{\lambda^{-}}$, $\nu$ is the unique minimizer of the quadratic perturbation $\Phi(\tau, \mu; \cdot)$ if and only if
\begin{align}
\frac{1}{\tau}(\mathbf{t}_{\nu}^\mu - \id) \in \partial E (\nu) \text{
is a strong subdifferential.} \label{eleqn}
\end{align}
Hence, $J_\tau \mu$ is
characterized by the fact that $\frac{1}{\tau}(\bt_{J_\tau \mu}^\mu - \id)
\in \partial E(J_\tau \mu)$ is a strong subdifferential.
\end{thm}

We assume $\mu \in D(E)$ and $E$ satisfies domain assumption \ref{small set assumption} to ease notation. See Theorem \ref{genelthm} for how the assumption $\mu \in D(E)$ can be relaxed to $\mu \in \overline{D(E)}$ and the domain assumption can be removed.

\begin{proof}[Proof of Theorem \ref{elthm}]
Suppose $\frac{1}{\tau}(\mathbf{t}_{\nu}^\mu - \id) \in \partial E (\nu)$ is a strong subdifferential. By Lemma \ref{subdifflemma}, we have ${\frac{1}{\tau}(\id - \mathbf{t}_\mu^{\nu}) \in \partial_{2,\mu} E(\nu)}$.
By additivity of the subdifferential and Proposition
\ref{W2rhosubdiffprop},
\[ \frac{1}{2\tau} 2(\mathbf{t}_\mu^\nu - \id) + \frac{1}{\tau}(\id -
\mathbf{t}_\mu^{\nu}) = 0 \in \partial_{2,\mu} \Phi(\tau, \mu; \nu) \ . \]
Since $W_2^2(\mu, \cdot) $ is 2-convex in the
$W_{2,\mu}$ metric and $E$ is $\lambda$-convex in the $W_{2, \mu}$ metric, $\Phi(\tau, \mu; \cdot)$ is $\left(\frac{1}{\tau} + \lambda \right)$-convex in the $W_{2,\mu}$ metric, with $\left(\frac{1}{\tau} + \lambda \right) > 0$. Therefore, by Remark
\ref{minimizersubdiffcor}, $0 \in \partial_{2,\mu} \Phi(\tau, \mu;
\nu)$ implies $\nu$ minimizes $\Phi(\tau, \mu; \cdot)$.
See \cite[Lemma 10.1.2]{AGS} for the converse.
\end{proof}

\section{Exponential Formula for the Wasserstein Metric} \label{exp form section}
We now apply the results on transport metrics and the Euler-Lagrange equation to give a new proof of the exponential formula in the style of Crandall and Liggett \cite{CrandallLiggett}.
 In section \ref{contraction inequality section}, we prove a new ``almost'' contraction inequality, analogous to the key inequality $\|J_\tau u - J_\tau v \| \leq \|u - v\|$ from Crandall and Liggett's proof. In section \ref{prox map time step section}, we apply our Euler-Lagrange equation to control the behavior of the discrete gradient flow $J^n_\tau \mu$ as the time step $\tau$ varies. In particular, proximal maps with different time steps can be related by considering intermediate points along a geodesic between $\mu$ and $J_\tau \mu$. The convexity of the square transport metric $W^2_{2,\mu}$ along this geodesic allows us to control the behavior of the geodesic in terms of its endpoints. In sections \ref{asymrecineqsec} and \ref{rasindsec}, we use these ideas to bound the distance between discrete gradient flow sequences with different time steps via an asymmetric induction in the style of Rasmussen \cite{R}. Finally, in section \ref{exp form subsection}, we conclude that the discrete gradient flow converges to the gradient flow.
We close section \ref{exp form subsection} by applying our estimates to give simple proofs of properties of the gradient flow, including the contracting semigroup property and the energy dissipation inequality.
(Note that we do not consider gradient flow with respect to the transport metrics, but instead use the transport metrics for intermediate estimates of the Wasserstein discrete gradient flow.)

\subsection{Almost Contraction Inequality} \label{contraction inequality section}
In this section, we use the convexity of $\Phi(\tau, \mu; \cdot)$ along generalized geodesics with base $\mu$, in the form of inequality (\ref{discevi}), to prove an almost contraction inequality for the discrete gradient flow. (Inequality (\ref{discevi}) is sufficient for this purpose---we use the stronger version in Theorem \ref{sdiscevi} later.)
Our approach is similar to previous work of Carlen and the author \cite{CC}, though instead of symmetrizing the contraction inequality, we leave the inequality in an asymmetric form that is compatible with the asymmetric induction in Theorems \ref{recineqthm}-\ref{W2RasBoundThm}.

For the $\lambda \leq 0$ case, our argument follows the first three steps in the proof of \cite[Lemma 4.2.4]{AGS}. For the $\lambda > 0$ case, we use a new approach.

\begin{thm}[almost contraction inequality] \label{cothm}
Suppose $E$ satisfies convexity assumption \ref{funas}, $\mu \in D(|\partial E|)$, and $\nu \in \overline{D(E)}$. Then we have the following inequalities for all $0< \tau< -\frac{1}{\lambda}$,
\begin{align*} 
\text{if $\lambda > 0$: } \quad (1+\lambda \tau)^2 W_2^2(J_\tau \mu, J_\tau \nu) &\leq W_2^2(\mu,\nu) + \tau^2 |\partial E|^2(\mu) + 2 \lambda \tau^2 \left[E(\nu) - \inf E\right] \ , \\
\text{if $\lambda \leq 0$: } \quad (1+\lambda \tau)^2 W_2^2(J_\tau \mu,J_\tau \nu)   &\leq W_2^2(\mu,\nu) + \tau^2 |\partial E|^2(\mu) \ .
\end{align*}
\end{thm}
\begin{proof}
Throughout the proof we abbreviate $J_\tau \mu$ by $\mu_\tau$ and $J_\tau \nu$ by $\nu_\tau$. By inequality (\ref{discevi}),
\begin{align} \label{conpf1}
(1+\lambda \tau) W_2^2(\mu_\tau,\nu_\tau) - W_2^2(\mu, \nu_\tau) \leq 2 \tau [ E(\nu_\tau) - E(\mu_\tau) - (1/2\tau) W_2^2(\mu,\mu_\tau)\ \ , \\
 \label{conpf2}
(1+\lambda \tau) W_2^2(\nu_\tau,\mu) - W_2^2(\nu, \mu)  \leq 2 \tau [ E(\mu) - E(\nu_\tau) - (1/2 \tau) W_2^2(\nu,\nu_\tau) ] \ .
\end{align}
Suppose $\lambda >0$. Dropping $-\frac{1}{2 \tau} W_2^2(\nu,\nu_\tau)$ from (\ref{conpf2}), multiplying (\ref{conpf1}) by $(1+\lambda \tau)$, and adding the two together,
\begin{align*}
(1+\lambda \tau)^2 W_2^2(\mu_\tau,\nu_\tau) - W_2^2(\mu, \nu) &\leq 2 \tau \left[ (1+\lambda \tau) E(\nu_\tau) -  E(\nu_\tau) + E(\mu) - (1+\lambda \tau) \left[E(\mu_\tau) + (1/2 \tau) W_2^2(\mu,\mu_\tau) \right]\right]
\end{align*}
Since $\lambda >0$, $E$ is bounded below \cite[Lemma 2.4.8]{AGS}. Applying inequality (\ref{theorem AGS1}) and the fact that $E(\nu_\tau) \leq E(\nu)$,
\begin{align*}
(1+\lambda \tau)^2 W_2^2(\mu_\tau,\nu_\tau) - W_2^2(\mu, \nu) &\leq \tau^2 \left[\frac{1}{1+\lambda \tau} |\partial E|^2(\mu) +2   \lambda  [E(\nu_\tau)- \inf E] \right]
\leq \tau^2 [|\partial E|^2(\mu) + 2 \lambda \left[E(\nu) - \inf E\right] ] ,
\end{align*}
which gives the result.

Now suppose $\lambda \leq 0$. Adding (\ref{conpf1}) and (\ref{conpf2}) and applying inequality (\ref{theorem AGS1}),
\begin{align} \label{conpf3}
(1+\lambda \tau)W_2^2(\mu_\tau,\nu_\tau) -W_2^2(\nu, \mu) +\lambda \tau W_2^2(\nu_\tau, \mu)   &\leq 2 \tau \left[ E(\mu) - E(\mu_\tau) - (1/2 \tau) W_2^2(\mu,\mu_\tau)\right] - W_2^2(\nu,\nu_\tau) \nonumber \\
&\leq \frac{\tau^2}{1+\lambda \tau} |\partial E|^2(\mu) - W_2^2(\nu,\nu_\tau) \ .
\end{align}
For $a,b>0$ and $0 < \epsilon < 1$, the convex function $\phi(\epsilon) = a^2/\epsilon + b^2/(1-\epsilon)$ attains its minimum $(a+b)^2$ at $\epsilon =a/(a+b)$, hence $a^2/\epsilon + b^2/(1-\epsilon) \geq (a+b)^2 $.

Consequently, with $\epsilon = -\lambda \tau $, we obtain
\begin{align} \label{conpf4}
W_2^2(\nu_\tau, \mu) &\leq (W_2(\nu_\tau,\nu) + W_2(\nu,\mu))^2 \leq - W_2^2(\nu_\tau, \nu)/\lambda \tau +  W_2^2(\nu,\mu) /(1+\lambda \tau)\ .
\end{align}
Multiplying by $-\lambda \tau$, adding to (\ref{conpf3}), and multiplying by $(1+\lambda \tau)$, we obtain the result
\begin{align*}
(1+\lambda \tau)^2 W_2^2(\mu_\tau,\nu_\tau)   &\leq W_2^2(\mu,\nu) + \tau^2 |\partial E|^2(\mu)\ .
\end{align*}
\end{proof}

Iterating the inequalities in the above theorem and applying inequality (\ref{theorem AGS1}) gives the following corollary.
\begin{cor} \label{iterated cothm}
Suppose $E$ satisfies convexity assumption \ref{funas}, $\mu \in D(|\partial E|)$, and $\nu \in \overline{D(E)}$. Then we have the following inequalities for all $0< \tau< -\frac{1}{\lambda}$,
\begin{align} \label{iterated cothm1}
\text{if $\lambda > 0$: } \quad W_2^2(J^n_\tau \mu, J^n_\tau \nu)
&\leq (1+\lambda \tau)^{-2n} W_2^2(\mu,\nu) + n \tau^2 \left( |\partial E|^2( \mu) + 2 \lambda \left[E(\nu) - \inf E \right] \right) \ , \\
\text{if $\lambda \leq 0$: } \quad W_2^2(J^n_\tau \mu, J^n_\tau \nu) &\leq (1+\lambda \tau)^{-2n} W_2^2(\mu,\nu) + n \tau^2(1+\lambda \tau)^{-2n} |\partial E|^2( \mu) \ .  \label{iterated cothm2}
\end{align}
\end{cor}


\subsection{Proximal Map with Large vs. Small Time Steps} \label{prox map time step section}

We now apply the Euler-Lagrange equation, Theorem \ref{elthm},  to relate the proximal map with a large time step $\tau$ to the proximal map with a small time step $h$.

\begin{thm} \label{proxmapdifftimethm}
Suppose $E$ satisfies convexity assumption \ref{funas}. Then if $\mu \in D(E)$ and ${0 < h \leq \tau< \frac{1}{\lambda^-}}$,
\[ J_\tau \mu = J_h \left[ \left( \frac{\tau-h}{\tau}
\mathbf{t}_\mu^{J_\tau \mu} + \frac{h}{\tau} \id \right) \# \mu \right] \]
\end{thm}
%
\begin{proof}[Proof of Theorem \ref{proxmapdifftimethm}]
For simplicity of notation, we prove the result in the case $E$ satisfies domain assumption \ref{small set assumption}. See Theorem \ref{proxmapdifftimethm general} for the general result. As in the previous proof, we abbreviate $J_\tau \mu$ by $\mu_\tau$.

By Theorem \ref{elthm}, $\bxi = \frac{1}{\tau}(\bt_{\mu_\tau}^\mu - \id ) \in \partial E(\mu_\tau)$
is a strong subdifferential. Since $h/\tau \leq 1$,
\begin{align} \label{proxmapeq1}
 (\id + h \bxi) = \left(\id + \frac{h}{\tau}(\bt_{\mu_\tau}^\mu - \id) \right) =
\left(\frac{\tau - h}{\tau} \id + \frac{h}{\tau} \bt_{\mu_\tau}^\mu \right) \ .
\end{align}
is cyclically monotone. Consequently, if we define $\nu =  (\id + h \bxi) \# \mu_\tau$,  then $ \bt_{\mu_\tau}^\nu =  \id + h \bxi$.
Rearranging shows $\frac{1}{h} (\bt_{\mu_\tau}^\nu - \id) = \bxi \in \partial E(\mu_\tau)$, so by a second application of Theorem \ref{elthm}, $\mu_\tau = \nu_h$.

We now rewrite $\nu$ as it appears in the theorem. By equation (\ref{proxmapeq1}), $(\id + h \bxi) =\left(\frac{\tau - h}{\tau} \id + \frac{h}{\tau} \bt_{\mu_\tau}^\mu \right) =  \left(\frac{\tau - h}{\tau}
\mathbf{t}_{\mu}^{\mu_\tau} + \frac{h}{\tau} \id \right) \circ \mathbf{t}_{\mu_\tau}^{\mu} $. Therefore, $\nu = (\id + h \bxi) \# \mu_\tau  =\left(\frac{\tau - h}{\tau} \mathbf{t}_{\mu}^{\mu_\tau} + \frac{h}{\tau}
\id \right) \# \mu$.
\end{proof}
After proving Theorem \ref{proxmapdifftimethm}, we discovered another proof of the same result, independently obtained by Jost and Mayer \cite{Jost, Mayer}. It is non-variational and quite different from the proof given above, and we hope our proof is of independent interest. 

\subsection{Asymmetric Recursive Inequality} \label{asymrecineqsec}
The following inequality bounds the Wasserstein distance between discrete gradient flow sequences with different time steps in terms of a convex combination of earlier elements of the sequences, plus a small error term. 
A fundamental difference between Crandall and Liggett's recursive inequality and our Theorem \ref{recineqthm} is that theirs involves the distance while ours involves the square distance. (This is a consequence of the fact that our contraction inequality Theorem \ref{cothm} involves the square distance plus error terms.) Therefore, where Crandall and Liggett were able to use the triangle inequality to control the distance to intermediate points on a geodesic in terms of the distance to its endpoints, we have to use the convexity of the square transport metric. The bulk of our proof is devoted to passing from the transport metric back to the Wasserstein metric.
\begin{thm} \label{recineqthm} Suppose $E$ satisfies convexity assumption \ref{funas}. Then if $\mu \in D(|\partial E|)$ and ${0 < h \leq \tau< \frac{1}{\lambda^-}}$,
\begin{align*}
&(1-\lambda^-h)^2 W_2^2(J_\tau^n \mu, J_h^m \mu) \\
&\quad \leq  \frac{h}{\tau} (1-\lambda^- \tau)^{-1} W_2^2(J^{n-1}_\tau \mu, J^{m-1}_h \mu)+ \frac{\tau -h}{\tau}W_{2}^2(J^{m-1}_h \mu,J^n_\tau \mu) + 2 h^2 (1-\lambda^- h)^{-2m}|\partial E|^2(\mu) \ .
\end{align*}
\end{thm}

To consider $\lambda \geq 0$ and $\lambda <0$ jointly, we replace $\lambda$ by $-\lambda^-$: any function that is $\lambda$ convex is also $-\lambda^-$ convex.

\begin{proof}
To simplify notation, we abbreviate $J^n_\tau \mu$ by $J^n$, $J^m_h \mu$ by $J^m$, and $[([\tau-h]/\tau)\bt_{J^{n-1}}^{J^n} + (h/\tau) \id] \# J^{n-1}$ by $\mu_{\frac{\tau-h}{\tau}}^{J^{n-1} \to J^n}$, since the latter is the geodesic from $J^{n-1}$ to $J^n$ at time $(\tau-h)/\tau$. With this, we have
\begin{align*}
(1-\lambda^- h)^2 W_2^2(J^n, J^m ) &= (1-\lambda^- h)^2  W_2^2( J_h
(\mu_{\frac{\tau-h}{\tau}}^{J^{n-1} \to J^n}) , J^m) &\text{by Theorem \ref{proxmapdifftimethm},}\\
&\leq W_2^2( \mu_{\frac{\tau-h}{\tau}}^{J^{n-1} \to J^n}
, J^{m-1} ) + h^2 |\partial E|^2(J^{m-1}) &\text{by Theorem \ref{cothm},} \\
&\leq W_{2, J^{n-1}} ^2( \mu_{\frac{\tau-h}{\tau}}^{J^{n-1} \to J^n}
, J^{m-1}) + h^2|\partial E|^2(J^{m-1}) &\text{by Remark \ref{transport metric bounds W2}}  .
\end{align*}
By Proposition \ref{convexity of pseudo metrics}, $W_{2, J^{n-1}}$ is convex along generalized geodesics with base $J^{n-1}$. In particular, it is convex along the geodesic $\mu_{\frac{\tau-h}{\tau}}^{J^{n-1} \to J^n}$, which gives
\begin{align} \label{asyrecpf1}
(1-\lambda^- h)^2 W_2^2(J^n, J^m ) &\leq \frac{h}{\tau} W_{2, J^{n-1}}^2(J^{n-1},J^{m-1}) + \frac{\tau -
h}{\tau}W_{2, J^{n-1}}^2(J^{m-1},J^n) + h^2|\partial E|^2(J^{m-1})\ .
\end{align}

By Remark \ref{transport metric bounds W2}, $W_{2, J^{n-1}}^2(J^{n-1},J^{m-1}) = W_{2}^2(J^{n-1},J^{m-1})$. To control the second term, we claim that
\begin{align} \label{asyrecpf2}
W_{2, J^{n-1}}^2(J^{m-1}, J^n) &\leq \frac{\tau}{h} \left(W_{2}^2(J^{m-1}, J^n) - (1-\lambda^- h) W_{2, J^{m-1}}^2(J^m, J^n) \right) \nonumber \\
&\quad + \frac{1}{1-\lambda^- \tau} W_2^2(J^{n-1},J^{m-1})+  
\frac{\tau h}{1-\lambda^- h} |\partial E|^2(J^{m-1}) \ .
\end{align}
Substituting (\ref{asyrecpf2}) into (\ref{asyrecpf1}), simplifying and rearranging,  and using $(1-\lambda^- h)^2 \leq (1-\lambda^- h)$ gives
\begin{align*}
(1-\lambda^- h)^2 (\tau/h)W_2^2(J^n , J^m ) \leq   \frac{1- \lambda^- h}{1-\lambda^- \tau} W_2^2( J^{n-1},J^{m-1})+ \frac{\tau -h}{h}W_{2}^2( J^{m-1},J^n) + \left[\frac{h(\tau - h)}{1- \lambda^- h} + h^2 \right] |\partial E|^2(J^{m-1}) \ .\end{align*} 
Multiplying by $h/\tau$ and using both $1-\lambda^- h  \leq 1$ and $|\partial E|^2(J^{m-1}) \leq (1-\lambda^-h)^{-2(m-1)}|\partial E|^2(\mu)$ gives the result
 \begin{align*}
(1-\lambda^- h)^2 W_2^2(J^n , J^m ) \leq \frac{h}{\tau} \frac{1}{1-\lambda^- \tau} W_2^2(J^{n-1},J^{m-1})+ \frac{\tau -h}{\tau}W_{2}^2(J^{m-1},J^n) + 2 h^2 (1-\lambda^-h)^{-2m} |\partial E|^2(\mu) \ .
\end{align*}

It remains to show (\ref{asyrecpf2}). Replacing $(\mu,\nu)$ in Theorem \ref{sdiscevi} with $(J^{m-1}, J^n)$ and $( J^{n-1},  J^{m-1})$ gives
\begin{align}
(1-\lambda^- h)W_{2,J^{m-1}}^2(J^m, J^n)-W_2^2(J^{m-1},J^n)  &\leq 2h \left[ E(J^n)- E(J^m) - W_2^2(J^{m-1},J^m)/(2h) \right] \ , \label{asympf1} \\
(1-\lambda^- \tau) W_{2,J^{n-1}}^2(J^n, J^{m-1})-W_2^2(J^{n-1},J^{m-1})  &\leq 2 \tau \left[ E(J^{m-1})- E(J^n) - W_2^2(J^{n-1},J^n)/(2 \tau) \right] \ . \label{asympf2}
\end{align} 
Multiplying (\ref{asympf1}) by $\tau$, (\ref{asympf2}) by $h$, adding them together, and applying inequality (\ref{theorem AGS1}) gives
\begin{align}\label{mixedfimestepsvarineq} 
&\tau (1-\lambda^- h) W_{2,J^{m-1}}^2(J^m, J^n) + h (1-\lambda^- \tau) W_{2, J^{n-1}}^2(J^n, J^{m-1})\nonumber \\
&\quad \leq \tau W_2^2(J^{m-1},J^n) + h W_2^2(J^{n-1},J^{m-1}) + 2 \tau h \left[E(J^{m-1}) - E(J^m) -W_2^2(J^{m-1},J^m) /(2h)\right] - h W_2^2(J^{n-1},J^n)  \nonumber \\
&\quad \leq \tau W_2^2(J^{m-1},J^n) + h W_2^2(J^{n-1},J^{m-1}) + \frac{\tau h^2}{1 - \lambda^- h} |\partial E|^2(J^{m-1})  - hW_2^2(J^{n-1},J^n)\ .
\end{align}
As in equation (\ref{conpf4}) we have, $\lambda^- \tau W_{2,J^{n-1}}^2(J^{m-1},J^n) \leq W_2^2(J^n, J^{n-1}) + \frac{ \lambda^- \tau}{1-\lambda^- \tau} W_2^2(J^{n-1}, J^{m-1})$.
Multiplying by $h$ and adding to (\ref{mixedfimestepsvarineq}) gives
\begin{align*}
&\tau (1-\lambda^- h) W_{2,J^{m-1}}^2(J^m, J^n) + h  W_{2, J^{n-1}}^2(J^n, J^{m-1})\nonumber \\
&\quad \leq \tau W_2^2(J^{m-1},J^n) + \frac{h}{1-\lambda^- \tau} W_2^2(J^{n-1},J^{m-1}) + \frac{\tau h^2}{1 - \lambda^- h} |\partial E|^2(J^{m-1})  \ .
\end{align*}
Rearranging and dividing by $h$ gives the desired bound (\ref{asyrecpf2}).

\end{proof}

\subsection{Distance Between Discrete Gradient Flows with Different Time Steps} \label{rasindsec}
The following bound controls the distance between discrete gradient flow sequences with time steps $\tau$ and $h$. It is inspired by Rasmussen's simplification of Crandall and Liggett's method \cite{R,Y}. Unlike in the Banach space case, we work with the square distance instead of the distance itself. While this complicated matters in the previous theorem, in simplifies the induction in the following theorem.

We begin with the base case of the induction, bounding the distance between the $0$th and $n$th terms.
\begin{lem} \label{lingrowthm}
Suppose $E$ satisfies convexity assumption \ref{funas}. Then if $\mu \in D(|\partial E|)$ and $0< \tau< \frac{1}{\lambda^-}$,
\begin{align*}
W_2(J^n_\tau \mu, \mu) &\leq \frac{n \tau}{(1-\tau \lambda^-)^n} |\partial E(\mu)|
\end{align*}
\end{lem}
\begin{proof}
This follows from the triangle inequality, (\ref{theorem AGS1}), and the inequalities $\frac{1}{1 + \tau \lambda }~\leq~\frac{1}{1-\tau \lambda^- }$ and $1\leq \frac{1}{1-\tau \lambda^- }$.
\begin{align*}
W_2(J^n_\tau \mu, \mu) \leq \sum_{i=1}^n W_2(J_\tau^i \mu, J^{i-1}_\tau
\mu) \leq \sum_{i=1}^n \frac{\tau}{1+\tau \lambda} |\partial E(J_\tau^{i-1} \mu)| \leq \sum_{i=1}^n \frac{\tau}{(1+\tau \lambda)^i} |\partial E(\mu)| \leq \frac{n \tau}{(1-\tau \lambda^-)^n} |\partial E(\mu)| \ .
\end{align*}
\end{proof}
 
\begin{thm} \label{W2RasBoundThm} Suppose $E$ satisfies convexity assumption \ref{funas}. Then if $\mu \in D(|\partial E|)$ and $0 < h \leq \tau< \frac{1}{\lambda^-}$,
\begin{align} \label{W2RasBound}
W_2^2(J_\tau^n \mu, J_h^m \mu) \leq \left[(n\tau - mh)^2 + \tau h m +
2 \tau^2 n \right] (1-\lambda^- \tau)^{-2n} (1- \lambda^- h)^{-2m} |\partial E|^2(\mu) \ .
\end{align}
\end{thm}

\begin{proof}
We proceed by induction. The base case, when either $n=0$ or $m=0$, follows from  Lemma \ref{lingrowthm}. We assume the inequality holds for
$(n-1,m)$ and $(n,m)$ and show that this implies it holds for $(n, m+1)$. 

By Theorem \ref{recineqthm},
\begin{align*}
&(1-\lambda^- h)^{2}W_2^2(J_\tau^n \mu, J_h^{m+1} \mu) \\
&\leq  \frac{h}{\tau} (1-\lambda^- \tau)^{-1} W_2^2(J_\tau^{n-1} \mu, J^{m}_h \mu) + \frac{\tau-
h}{\tau} W_2^2(J^m_h \mu,J^n_\tau \mu) +  2h^2(1-\lambda^- h)^{-2(m+1)} |\partial E|^2(\mu) \ .
\end{align*}
Dividing by $(1-\lambda^- h)^{2}$ and applying the inductive hypothesis,
\begin{align*}
W_2^2(J_\tau^n \mu, J_h^{m+1} \mu) &\leq   \frac{h}{\tau}  \left[((n-1)\tau - mh)^2 + \tau h m+ 2\tau^2 (n-1) \right]  (1-\lambda^- \tau)^{-2(n-1)-1}(1-\lambda^- h)^{-2(m+1)} |\partial E|^2(\mu) \\
&\quad + \frac{\tau-h}{\tau} \left[(n\tau - mh)^2 + \tau h m + 2 \tau^2 n \right] (1-\lambda^- \tau)^{-2n}(1-\lambda^- h)^{-2(m+1)} |\partial E(\mu)|^2 \\
&\quad +2h^2(1-\lambda^- h)^{-2(m+1)-2}   |\partial E|^2(\mu) \ .
\end{align*}
To control the first term, note that $ (1-\lambda^- \tau)^{-2(n-1)-1}=(1-\lambda^- \tau)^{-2n+1} <  (1-\lambda^- \tau)^{-2n}$ and 
\begin{align*}
 \left[((n-1)\tau - mh)^2 + \tau h m+ 2 \tau^2 (n-1) \right] = \left[ (n\tau - mh)^2 - 2(n \tau - mh)\tau + \tau^2 + \tau h m+ 2 \tau^2 (n-1) \right] \ .
\end{align*}
To control the third term, note that $(1-\lambda^- h)^{-2} \leq (1-\lambda^- \tau)^{-2} \leq (1-\lambda^- \tau)^{-2n}$.
Using these estimates, we may group together the three terms and obtain the following bound.
\begin{align*}
 &W_2^2(J_\tau^n \mu, J_h^{m+1} \mu) \\
 &\leq \left\{ \frac{h}{\tau} \left[ (n\tau - mh)^2 - 2(n \tau - mh)\tau +
\tau^2 + \tau h m+ 2\tau^2(n-1) \right]+ \frac{\tau-h}{\tau} \left[(n\tau -
mh)^2 + \tau h m+ 2\tau^2 n \right] +2h^2 \right\} \\
&\quad \cdot (1-\lambda^- \tau)^{-2n}(1-\lambda^- h)^{-2(m+1)}|\partial E|^2(\mu) \ .
\end{align*}
Simplifying and bounding the quantity within the brackets gives the result,
\begin{align*}
& \frac{h}{\tau} \left[ (n\tau - mh)^2 - 2(n \tau - mh)\tau +
\tau^2 + \tau h m+ 2 \tau^2(n-1)\right]+ \frac{\tau-h}{\tau} \left[(n\tau -
mh)^2 + \tau h m+ 2\tau^2 n \right] +2h^2 \\
&=\frac{h}{\tau} \left[ (n\tau - mh)^2 + \tau h m+ 2 \tau^2 n \right]+ \frac{\tau-h}{\tau} \left[(n\tau -mh)^2 + \tau h m+ 2\tau^2 n \right] + \frac{h}{\tau} \left[- 2(n \tau - mh)\tau -\tau^2\right]  +2h^2 \\
&=\left[ (n\tau - mh)^2 + \tau h m+ 2 \tau^2 n \right]- 2(n \tau - mh)h -\tau h  +2h^2 \\
&=(n \tau - m h)^2 - 2(n \tau - mh) h +\tau h m - \tau h + 2 \tau^2 n + 2 h^2 \\
&=(n \tau - (m+1)h)^2  + \tau h m - \tau h +2 \tau^2 n +h^2\\
&\leq (n \tau - (m+1)h)^2 + \tau h (m+1) + 2 \tau^2 n \ .
\end{align*}
\end{proof}

\subsection{Exponential Formula for the Wasserstein Metric} \label{exp form subsection}
We now combine our previous results to conclude the exponential formula for the Wasserstein metric. 
We prove the quantitative bound $W_2(J_{t/n}^n \mu , \mu(t)) \leq O(n^{-1/2} )$, which agrees with the rate Crandall and Liggett obtained in the Banach space case\cite{CrandallLiggett}. By a different method, Ambrosio, Gigli, and Savar\'e showed ${W_2(J_{t/n}^n \mu ,\mu(t)) \leq O(n^{-1})}$ \cite[Theorem 4.0.4]{AGS}, which agrees with the optimal rate in a Hilbert space \cite{Rulla}.
Our rate improves upon the rate obtained by Cl\'ement and Desch  \cite{ClementDesch}, $d(J_{t/n}^n \mu , \mu(t))~\leq~O(n^{-1/4} )$, who also pursued a Crandall and Liggett type approach, though they considered the more general case of a metric space $(X,d)$ that satisfies inequality (\ref{discevi}), with $W_2$ replaced by $d$.

Though we do not obtain the optimal rate of convergence, we demonstrate that Crandall and Liggett's approach extends to the Wasserstein metric, providing a simple and robust route to the exponential formula and properties of continuous gradient flow. In particular, it is hoped that this method may be used to study the behavior of the gradient flow as the functional $E$ varies---for example, as a regularization of $E$ is removed. The corresponding problem in the Banach space case is well-understood \cite{BrezisPazy}, and the analogy we establish between the Banach and Wasserstein cases may help extend these results.

\begin{thm}[exponential formula] \label{exp form thm}
Suppose $E$ satisfies convexity assumption \ref{funas}. For ${\mu \in \overline{D(E)}}$, $t \geq 0$, the discrete gradient flow sequence $J^n_{t/n} \mu$ converges as $n \to \infty$. The convergence is uniform in $t$ on compact subsets of $[0, +\infty)$, and the limit $\mu(t)$ is the gradient flow of $E$ with initial conditions $\mu$ in the sense of Definition \ref{gradflowdef}. When $\mu \in D(|\partial E|)$ and $n > 2 \lambda^-t$, the distance between $J^n_{t/n}$ and $\mu(t)$ is bounded by
\begin{align}\label{errest}
 W_2( J^n_{t/n}\mu, \mu(t)) \leq \sqrt{3} \frac{t}{\sqrt{n}}  e^{3 \lambda^- t} |\partial E|(\mu) \ .
 \end{align}
\end{thm}

\begin{remark}[varying time steps]
For any partition of the interval $[0,t]$ into $n$ time steps $\tau_1, \dots, \tau_n$, the corresponding discrete gradient flow with varying time steps $\Pi_{i=1}^n J_{\tau_i} \mu$ converges to $\mu(t)$ as the maximum time step goes to zero. See Theorem \ref{exp form thm var time}.
\end{remark}

We may apply our estimates to give shorts proofs of several known properties of the gradient flow \cite[Theorem 2.4.15, Proposition 4.3.1, Corollary 4.3.3]{AGS}. Our proofs merely use the fact that $\mu(t) = \lim_{n\to \infty} J_{t/n}^n \mu$, not that $\mu(t)$ is a solution to the gradient flow in the sense of Definition \ref{gradflowdef}. 

\begin{thm} \label{contracting semigroup 1}
Suppose $E$ satisfies convexity assumption \ref{funas}. Then the function $S(t)$ on $(0,+\infty)$, $S(t): \overline{D(E)} \to D(|\partial E|) : \mu \mapsto \mu(t)$  is a $\lambda$-contracting semigroup, i.e.
\begin{enumerate}[(i)]
	\item $\lim_{t\to 0} S(t) \mu = S(0) \mu = \mu$,
	\item $S(t+s) = S(t) S(s) \mu$ for $t, s \geq 0$,
	\item $W_2(S(t)\mu,S(t)\nu) \leq e^{-\lambda t} W_2(\mu,\nu)$.
\end{enumerate}
Furthermore, $\mu(t)$ is locally Lipschitz on $(0,+\infty)$ and if $\mu \in D(|\partial E|)$,
\begin{align} \label{time Lipschitz}
W_2(\mu(t), \mu(s)) \leq |t-s| e^{\lambda^- t} e^{ \lambda^- s} |\partial E|(\mu) \ .
\end{align}
Finally, an energy dissipation inequality holds,
\begin{align} \label{energydissipationineq}
 \int_{t_0}^{t_1} |\partial E|^2(\mu(s)) ds \leq E(\mu(t_0)) - E( \mu(t_1)) \ .
 \end{align}
\end{thm}

We now turn to the proofs of these results.

\begin{proof}[Proof of Theorem \ref{exp form thm}]
Throughout, we abbreviate $J_\tau \mu$ by $\mu_\tau$ and use that $\mu \in \overline{D(E)}$ implies $\mu_\tau \in D(|\partial E|)$ by (\ref{theorem AGS1}) and $\lim_{\tau \to 0} \mu_\tau = \mu$ by \cite[Lemma 3.1.2]{AGS}.

First, we prove the error estimate (\ref{errest}) for $\mu \in D(|\partial E|)$. By Theorem \ref{W2RasBoundThm}, if we define $\tau =
\frac{t}{n}$, $h = \frac{t}{m}$, with $m \geq n >2t \lambda^-$, so $0 \leq h \leq \tau <\frac{1}{2\lambda^-}$,
\begin{align} \label{CauchyEqn}
W_2^2(J_{t/n}^n \mu , J_{t/m}^m \mu) \leq 3 \frac{t^2 }{n} (1- \lambda^- t /n)^{-2n} (1-\lambda^- t/m)^{-2m} |\partial E|^2(\mu) \leq 3  \frac{t^2}{n} e^{8 \lambda^-t} |\partial E|^2(\mu) \ .
\end{align}
In the second inequality, we use $(1-\alpha)^{-1} \leq e^{2\alpha}$ for $\alpha \in [0, 1/2]$.
Thus, the sequence $J_{t/n}^n \mu$ is Cauchy, and $\lim_{n \to \infty}
J_{t/n}^n \mu$ exists \cite[Prop 7.1.5]{AGS}. The convergence is uniform in $t$ on compact subsets of $[0, +\infty)$. If $\mu(t)$ denotes the limit, then sending $m \to \infty$ in the first inequality of (\ref{CauchyEqn}) gives (\ref{errest}).

Now suppose $\mu \in \overline{D(E)}$. By the triangle inequality,
\begin{align*}
W_2(J^n_{t/n} \mu, J^m_{t/m} \mu) &\leq W_2(J^n_{t/n} \mu, J^n_{t/n} \mu_\tau) + W_2(J^n_{t/n} \mu_\tau, J^m_{t/m} \mu_\tau) + W_2(J^m_{t/m} \mu_\tau, J^m_{t/m} \mu) \ .
\end{align*}
By the previous estimate, we may choose $n, m$ large enough so that the second term arbitrarily small. By Corollary \ref{iterated cothm}, we may choose $n, m$ large enough and $\tau$ small enough so that the first and third terms are arbitrarily small.
Thus, the sequence $J^n_{t/n} \mu$ is Cauchy uniformly in $t \in [0,T]$, so the limit exists and convergence is uniform for $t \in [0,T]$.

It remains to show that $\mu(t)$ is the gradient flow of $E$ with initial conditions $\mu$ in the sense of Definition \ref{gradflowdef}. As our proof of Theorem \ref{contracting semigroup 1} merely uses the fact that $\mu(t) = \lim_{n\to \infty} J_{t/n}^n \mu$, we may leverage these results. By (i), $\lim_{t\to 0}S(t) \mu = \mu$, so it remains to show that $S(t) \mu$ satisfies inequality (\ref{continuous evi}) from Definition \ref{gradflowdef}.

Define the piecewise constant function $g_n(s) = [  E(J^i_{\tau} \mu)- E(\omega)](1+\lambda \tau)^{i-1}$ for $s \in((i-1)\tau, i\tau]$, $1\leq i \leq n$.
Iterating inequality (\ref{discevi}) for $0 < \tau < 1/\lambda^-$ shows that for all $ \omega \in D(E)$,
\begin{align*}
(1+\lambda \tau)^n W_2^2(J^n_\tau \mu, \omega) \leq W_2^2(\mu, \omega)+ 2 \tau \sum_{i=1}^n [E(\omega) - E(J^i_\tau \mu)] (1+ \lambda \tau)^{i-1}= W_2^2(\mu,\omega) - 2\int_0^{\tau n} g_n(s) ds
\end{align*}
Set $\tau = t/n$ and suppose the following claim holds: \\ \noindent
\textbf{Claim.} $\liminf_{n \to \infty}\int_0^t g_n(s) ds \geq \int_0^t [E(S(s) \mu) - E(\omega)]e^{\lambda s} ds$.

\noindent Then taking the liminf of the above inequality gives
\begin{align}
e^{\lambda t} W_2^2(S(t) \mu, \omega) \leq W_2^2(\mu,\omega) + 2 \int_0^t [E(\omega) - E(S(s) \mu)] e^{\lambda s} ds \ .
\end{align}
Applying the semigroup property $(ii)$ of Theorem \ref{contracting semigroup 1} and multiplying through by $e^{\lambda t_0}$,
\begin{align*}
e^{\lambda (t+t_0)} W_2^2(S(t +t_0) \mu, \omega) - e^{\lambda t_0}W_2^2(S(t_0)\mu,\omega) &\leq 2 \int_{t_0}^{t+t_0} [E(\omega) - E(S(s) \mu)] e^{\lambda s} ds \ .
\end{align*}
Dividing by $t$, sending $t \to 0$, and dividing by $2 e^{\lambda t_0}$ gives inequality (\ref{continuous evi}).

We conclude by proving the claim.
Since $E(J^i_{t/n} \mu) \geq E(J^n_{t/n} \mu)$ and $\liminf_{n \to \infty} E(J^n_{t/n} \mu) \geq E(S(t) \mu)$, $g_n(s)$ is bounded below. By Fatou's lemma, it is enough to show $ \liminf_{n \to \infty} g_{n}(s) ds \geq [E(S(s) \mu) - E(\omega)]e^{\lambda s}$, and by the lower semicontinuity of $E$, this holds if $\lim_{n \to \infty} J^i_{t/n} \mu = S(s) \mu$.
By the triangle inequality,
\[ W_2(J^i_{t/n}\mu, S(s)\mu) \leq W_2(J^i_{t/n}\mu, J^i_{t/n} \mu_\tau) + W_2(J^i_{t/n} \mu_\tau, S(s) \mu_\tau) + W_2(S(s) \mu_\tau, S(s) \mu) \ . \] 
By Corollary \ref{iterated cothm}, we may choose $n$ large enough and $\tau$ small enough, uniformly in $t \in [0,T]$, so that the first term is arbitrarily small. Likewise, the third term may be made arbitrarily small by the contraction inequality (iii) of Theorem \ref{contracting semigroup 1}. Thus it remains to show that $\lim_{n \to \infty} J^i_{t/n}\mu_\tau = S(s) \mu_\tau$. This follows by Theorem \ref{W2RasBoundThm}, since for $s \in \left((i-1)t/n, it/n \right]$ and $m, n$ large enough,
\begin{align*}
W_2^2(J^i_{t/{n}} \mu_\tau,J^m_{s/m} \mu_\tau) &\leq \left(\left(i \frac{t}{n} - s\right)^2 + \frac{t}{n}s + 2\frac{t^2}{n^2}i \right) e^{4 \lambda^- (t+s)} |\partial E|(\mu_\tau) \leq \left(\frac{3t^2}{n^2} + \frac{ts+ 2t^2}{n}\right) e^{4 \lambda^- (t+s)} |\partial E|(\mu_\tau) \ .
\end{align*}
 \end{proof}
 
\begin{proof}[Proof of Theorem \ref{contracting semigroup 1}]
Again, we abbreviate $J_\tau \mu$ by $\mu_\tau$ and use that $\mu \in \overline{D(E)}$ implies $\mu_\tau \in D(|\partial E|)$ by (\ref{theorem AGS1}) and $\lim_{\tau \to 0} \mu_\tau = \mu$ by \cite[Lemma 3.1.2]{AGS}.

(iii) follows by sending $n \to \infty$ in Corollary \ref{iterated cothm}. In particular, for $\mu, \nu \in \overline{D(E)}$,
\[W_2^2(S(t) \mu_\tau, S(t) \nu) \leq e^{-\lambda t} W_2^2(\mu_\tau,\nu) \ . \]
The result then follows by the triangle inequality and the continuity of the proximal map, as $\tau \to 0$,
\begin{align*}
W_2(S(t)\mu , S(t) \nu) &\leq W_2(S(t)\mu , S(t) \mu_\tau) + W_2(S(t)\mu_\tau , S(t) \nu) \leq e^{-\lambda t} W_2(\mu, \mu_\tau) + e^{-\lambda t}W_2(\mu_\tau, \nu) \ .
\end{align*}

Next, we show estimate (\ref{time Lipschitz}) on the modulus of continuity for $S(t) \mu$ when $ \mu \in D(|\partial E|)$. Given $t\geq s \geq 0$, define $\tau=\frac{t}{n}, h= \frac{s}{m}$ for $m$, $n$ large enough so $0 \leq h \leq \tau < \frac{1}{\lambda^-}$. By Theorem \ref{W2RasBoundThm},
 \begin{align}
W_2^2(J_{t/n}^n \mu, J_{s/m}^m \mu) \leq \left[(t - s)^2 + \frac{ts}{n} +
2 \frac{t^2}{n} \right] (1-\lambda^- t/n)^{-2n} (1- \lambda^- s/m)^{-2m} |\partial E|^2(\mu) \ .
\end{align}
Sending $m,n \to \infty$ and taking the square root of both sides gives (\ref{time Lipschitz}). If $\mu \in \overline{D(E)}$, then by (iii), $\lim_{\tau \to 0} S(t) \mu_\tau = S(t) \mu$ uniformly for $t \in [0,T]$, so $S(t) \mu$ is also continuous for $t \in [0,T]$. Thus (i) holds.

To show (ii), note that it is enough to show $S(t)^m \mu = S(mt) \mu$ for $m \in \mathbb{N}$. With this, we have for $l,k,r,s \in \mathbb{N}$,
\[ S\left(\frac{l}{k} + \frac{r}{s} \right) \mu = S\left(\frac{ls + rk}{ks} \right) \mu = \left[S \left(\frac{1}{ks} \right) \right]^{ls + rk}  \mu = \left[ S \left(\frac{1}{ks} \right) \right]^{ls} \left[ S \left(\frac{1}{ks} \right) \right]^{rk} \mu  = S\left(\frac{l}{k} \right) S\left( \frac{r}{s} \right) \mu  \ .\]
Since $S(t) \mu$ is continuous in $t \in [0, +\infty)$, $S(t+s) \mu= S(t) S(s) \mu$ for all $t, s \geq 0$. Likewise, it is enough to show the result for $\mu \in D(|\partial E|)$. If $\mu \in \overline{D(E)}$, then by (iii), $\lim_{\tau \to 0} S(t) \mu_\tau = S(t) \mu$ uniformly for $t \in [0,T]$, so $S(t+s) \mu = \lim_{\tau \to 0} S(t+s) \mu_\tau = \lim_{\tau \to 0} S(t)S(s) \mu_\tau = S(t) S(s) \mu$.

We proceed by induction. $S(t)^m \mu = S(mt) \mu$ for $m =1$. Suppose that $S(t)^{m-1} \mu = S((m-1)t) \mu$. We show that we may choose $n$ large enough to make the right hand side of the following inequality arbitrarily small:
\begin{align*}
&W_2(S(mt) \mu, S(t)^m \mu) \leq W_2(S(mt) \mu, (J^n_{t/n})^m \mu)  + W_2( (J^n_{t/n})^{m} \mu, J^n_{t/n} S(t)^{m-1} \mu)  + W_2( J^n_{t/n} S(t)^{m-1} \mu, S(t)^{m} \mu)
\end{align*}
Since $W_2( S(mt) \mu, (J^n_{t/n})^m \mu)  = W_2( S(mt) \mu,J^{nm}_{tm/nm} \mu)$ and $\lim_{n \to \infty} J^{nm}_{tm/nm} \mu = S(mt) \mu$, the first term may be made arbitrarily small. Since $\lim_{n \to \infty} J^n_{t/n} S(t)^{m-1} \mu = S(t)^m \mu$, so may the third term.

To bound the second term, note that by the lower semicontinuity of  $|\partial E|$ \cite[Corollary 2.4.10]{AGS} and (\ref{theorem AGS1}),
\begin{align} \label{subdiff bound along grad flow}
|\partial E|(\mu(t)) \leq \liminf_{n\to \infty} |\partial E|(J^n_{t/n} \mu) \leq \liminf_{n \to \infty} (1-\lambda^- t/n)^{-n} |\partial E|(\mu) = e^{\lambda^- t} |\partial E|(\mu) \ .
\end{align}
Hence, $|\partial E|^2( S(t)^{m-1} \mu) \leq e^{2(m-1)\lambda^- t} |\partial E|^2(\mu)$. Combining this with Corollary \ref{iterated cothm},
\begin{align*}
&W^2_2(J^n_{t/n} S(t)^{m-1} \mu, J^n_{t/n} (J^n_{t/n})^{m-1} \mu) \\
&\quad \leq(1-\lambda^- (t/n))^{-2n} W_2^2(S(t)^{m-1} \mu,(J^n_{t/n})^{m-1} \mu) + \frac{n (t/n)^2 e^{2(m-1)\lambda^- t}}{(1-\lambda^- (t/n))^{2n}} |\partial E|^2(\mu) \ .
\end{align*}
By the inductive hypothesis, $\lim_{n \to \infty} (J^n_{t/n})^{m-1} \mu = \lim_{n \to \infty} J^{n(m-1)}_{t(m-1)/n(m-1)} \mu = {S((m-1)t) \mu} = S(t)^{m-1} \mu$. Thus, we may choose $n$ large enough to make the second term arbitrarily small.

Finally, we show the energy dissipation inequality (\ref{energydissipationineq}).
By (ii), it is enough to prove the result for $t_0 =0$, $t_1 = t$. By inequality (\ref{theorem AGS1}),
\begin{align} \label{semigroup pf 2}
 \tau (1+\lambda \tau/2) |\partial E|^2(\mu_\tau) \leq E(\mu) - E(\mu_\tau) \ .
 \end{align}
  Define $g_n(s) = (1+\lambda \tau/2)|\partial E|^2(J^i_\tau \mu)$ for $s \in((i-1)\tau, i\tau ]$, $1\leq i \leq n$.
Summing (\ref{semigroup pf 2}),
\begin{align} \label{energydis1}
\int_0^{\tau n} g_n(s) ds =  \sum_{i=1}^{n} \tau(1+ \lambda \tau/2) |\partial E|^2(J^i_\tau \mu) \leq E(\mu) - E(J^n_\tau \mu)  \ .
 \end{align}
Let $\tau = t/n$. As in the proof of the previous theorem, $\lim_{n \to \infty} J^i_{t/n}\mu= S(s) \mu$. Taking the liminf of (\ref{energydis1}) and applying Fatou's lemma along with the lower semicontinuity of $E$ and $|\partial E|$ \cite[Corollary 2.4.10]{AGS} gives
\begin{align*}
\int_0^t |\partial E|^2(S(s) \mu) = \int_0^t \liminf_{n \to \infty} g_n(s) ds \leq E(\mu) - E(S(t) \mu) \ .
\end{align*}

The lower semicontinuity of $E$ and the fact that $J^n_{t/n} \mu \in D(E)$ ensure $E(S(t) \mu) < +\infty$ for all $t \geq 0$. Furthermore, the energy dissipation inequality implies that $|\partial E|^2 (S(s)\mu) < +\infty $ for almost every $s \geq 0$. By (iii) and (\ref{subdiff bound along grad flow}), $|\partial E|(S(t) \mu) \leq e^{\lambda^- (t-s)} |\partial E|(S(s) \mu)$ for $0<s< t$. Therefore, $S(t): \overline{D(E)} \to D(|\partial E|)$.
\end{proof}

\appendix{
\section{Generalizations}
\subsection{Varying Time Steps} \label{var time step section}
This section contains generalizations of the previous theorems to the case where we replace $m$ time steps of size $h$ with a sequence of varying time steps.
For simplicity of notation, we write $J^m = \prod_{k=1}^m J_{h_k} \mu$, $J^n= J^n_\tau \mu$, $S_m= \sum_{k=1}^m h_k$, and $P_m = \prod_{k=1}^{m} (1-\lambda^-h_k)^{-1}$.

In Theorems \ref{var tau asym rec ineq}-\ref{W2RasBoundvartime}  below, we suppose $E$ satisfies convexity assumption \ref{funas}, $\mu \in D(|\partial E|)$, and $0 < h_i \leq \tau< \frac{1}{\lambda^-}$.
The first result is a generalization of Theorem \ref{recineqthm}, the second is a generalization of Lemma \ref{lingrowthm}, and the third is a generalization of Theorem \ref{W2RasBoundThm}.

 \begin{thm}\label{var tau asym rec ineq} \ \\
$(1-\lambda^- h_m)^2 W_2^2(J^n , J^m ) \leq  \frac{h_m}{\tau} (1-\lambda^- \tau)^{-1} W_2^2( J^{n-1},J^{m-1})+ \frac{\tau -h_m}{\tau}W_{2}^2(J^{m-1}, J^n) + 2 h_m^2 P_m^2 |\partial E|^2(\mu)$.
\end{thm}

%
%

\begin{lem} \label{lingrowthm vartimestep}
$W_2 \left(J^m, \mu \right) \leq |\partial E(\mu)| S_m P_m$.
\end{lem}

\begin{thm}  \label{W2RasBoundvartime} 
$W_2^2(J^n, J^m) \leq \left[ \left( n\tau - S_m \right)^2 + \tau S_m +
2 \tau^2 n \right] (1-\lambda^- \tau)^{-2n} P_m^2 |\partial E|^2(\mu) $.
\end{thm}
%

The proof of each of these results is a straightforward generalization of the previous proof. (Simply replace $h$ with $h_m$, use $|\partial E|^2(J^{m})\leq P_m^2 |\partial E|^2(\mu)$, and, in the third result, replace $mh$ by $S_m$.)

We combine these results in the following theorem to prove the convergence of the discrete gradient flow with varying time steps to the continuous gradient flow. 
\begin{thm} \label{exp form thm var time}
Suppose $E$ satisfies convexity assumption \ref{funas}. For $\mu \in D(|\partial E|)$ and any partition of the interval $\{0 = t_0 < t_1 < \dots < t_i < t_{i+1} < \dots t_m = t \}$ corresponding to time steps $h_i = t_i - t_{i+1}$
the discrete gradient flow sequence sequence $\prod_{i=1}^m J_{h_i} \mu$ converges to the gradient flow $\mu(t)$ as $|\bh| = \max_{1\leq i\leq m} h_i \to 0$. The convergence is uniform in $t$ on compact subsets of $[0, +\infty)$. When $|\bh| \leq \frac{1}{2 \lambda^-}$, the distance between the approximating sequence $\prod_{i=1}^m J_{h_i} \mu$ and the continuous gradient flow $S(t) \mu$ is bounded by
\begin{align*}
W_2\left(S(t) \mu, \prod_{i=1}^m J_{h_i} \mu \right) \leq 2\left[ |\bh|^2 + 3|\bh| t \right]^{1/2} e^{4\lambda^-t}|\partial E|(\mu) 
\end{align*}
\end{thm}
\begin{proof}
Let $\tau= |\bh|$, so $0 < h_k \leq \tau< \frac{1}{2\lambda^-}$, and let $n$ be the greatest integer less than or equal to $t/\tau$, so $ t/\tau - 1 < n \leq t/\tau$, hence $t - \tau < n \tau \leq t$.
By the triangle inequality, Theorem \ref{W2RasBoundvartime}, and the fact that $(1-\alpha)^{-1} \leq e^{2\alpha}$ for $\alpha \in [0, 1/2]$,
\begin{align*}
W_2\left(S(t) \mu, \prod_{i=1}^m J_{h_i} \mu \right) &\leq W_2 (S(t) \mu, J_{\tau}^n \mu)+ W_2 \left(J_\tau^n \mu, \prod_{i=1}^m J_{h_i} \mu \right) =\lim_{l \to \infty} W_2 (J^l_{t/l} \mu, J_{\tau}^n \mu)+ W_2 \left(J_\tau^n \mu, \prod_{i=1}^m J_{h_i} \mu \right) \\
&\leq \lim_{l \to \infty} \left[(n\tau - t)^2 + \tau t +
2 \tau^2 n \right]^{1/2} (1-\lambda^- \tau)^{-n} (1- \lambda^- t/l)^{-l} |\partial E|(\mu) \\
&\quad + \left[ \left( n\tau - t \right)^2 + \tau t +
2 \tau^2 n \right]^{1/2} (1-\lambda^- \tau)^{-n} \left\{ \prod_{k=1}^m(1- \lambda^- h_k)^{-1} \right\} |\partial E|(\mu) \\
&\leq 2\left[ |\bh|^2 + 3|\bh| t \right]^{1/2} e^{4\lambda^-t}|\partial E|(\mu) 
\end{align*}
\end{proof}

\subsection{Allowing measures which give mass to small sets} \label{small set section}
In this section, we give proofs of the Euler-Lagrange equation (Theorem \ref{elthm}) and the relation between proximal maps with small and large time steps (Theorem \ref{proxmapdifftimethm}) in the general case, in which measures may give mass to small sets. In this context, optimal transport maps may no longer exist, so the transport metrics (Definition \ref{transport distance def}) may no longer be well-defined.

\begin{thm}[Euler-Lagrange equation] \label{genelthm} Suppose $E$  satisfies convexity assumption \ref{funas}.
Then for $\omega \in \overline{D(E)}$, $0< \tau< \frac{1}{\lambda^{-}}$, $\mu$ is the unique minimizer of the quadratic perturbation $\Phi(\tau, \omega; \cdot)$ if and only if for all $\gamma \in \Gamma_0(\mu,\omega)$, defining $\gamma_\tau = \rho_\tau \# \gamma$, $\rho_\tau(x_1,x_2) =(x_1, (x_2-x_1)/\tau)$, we have 
\begin{align} \label{gen EL pf 1} 
 E(\nu) - E(\mu) \geq \int \la x_2, x_3 - x_1 \ra d \bgamma + o\left( \|x_1 - x_3\|_{L^2(\bgamma)} \right)\quad  \forall \bgamma \in \Gamma(\gamma_\tau, \nu) , \nu \in D(E) , \nu \to \mu \ .
 \end{align}
\end{thm}

\begin{proof}
\cite[Lemma 10.3.4]{AGS} shows that if $\mu$ is the unique minimizer of $\Phi(\tau, \omega; \cdot)$, then (\ref{gen EL pf 1}) holds for all $\gamma \in \Gamma_0(\mu,\omega)$.

We now prove the converse. Suppose that for all $\gamma \in \Gamma_0(\mu,\omega)$, (\ref{gen EL pf 1}) holds, and fix $\nu \in D(E)$. There exists some generalized geodesic $\mu_\alpha$ from $\mu$ to $\nu$ with base $\omega$ along which $E$ is $\lambda$-convex. Let $\bomega$ be the plan that induces this generalized geodesic, with $\pi^{1,2} \# \bomega \in \Gamma_0(\omega,\mu)$ and $\pi^{1,3} \# \bomega \in \Gamma_0(\omega,\nu) $, so $\mu_\alpha = [(1-\alpha)\pi^2 + \alpha \pi^3] \# \bomega$. 
Applying (\ref{gen EL pf 1}) with$\gamma = \pi^{2,1} \# \bomega$ and $\nu = \mu_\alpha$ shows
\begin{align} \label{gen EL pf 2}
 E(\mu_\alpha) - E(\mu) \geq \int \la x_2, x_3 - x_1 \ra d \bgamma_\alpha + o \left(\| x_1 - x_3\|_{L^2(\bgamma_\alpha)}\right)\quad \forall \bgamma_\alpha \in \Gamma(\gamma_\tau, \mu_\alpha), \mu_\alpha \to \mu \ .
 \end{align}
Since $(x_2,\frac{x_1-x_2}{\tau},(1-\alpha)x_2 + \alpha x_3) \# \bomega \in \Gamma(\gamma_\tau, \mu_\alpha)$, 
\begin{align*}
 E(\mu_\alpha) - E(\mu) &\geq \int \left\la (x_1-x_2)/\tau, ((1-\alpha)x_2 + \alpha x_3) - x_2 \right\ra d \bomega + o \left(\| x_2 - (1-\alpha) x_2 - \alpha x_3\|_{L^2(\bomega)}\right) \nonumber \\
  &= \alpha \int \left\la (x_1-x_2)/\tau, x_3-x_2 \right\ra d \bomega + \alpha o \left( \| x_2 - x_3\|_{L^2(\bomega)}\right) \ .
 \end{align*}
By definition of convexity along a generalized geodesic, 
$ E(\nu)- E(\mu) \geq \frac{1}{\alpha} \left[E(\mu_\alpha) - E(\mu) \right] +  (1- \alpha)\frac{\lambda}{2} \| x_2 - x_3\|_{L^2(\bomega)}$. 
Using the above inequality, we may bound this from below,
 \begin{align*}
 E(\nu)- E(\mu) &\geq \int \left\la (x_1-x_2)/\tau, x_3-x_2  \right\ra d \bomega + o \left(1 \right)+  (1- \alpha)(\lambda/2) \| x_2 - x_3\|_{L^2(\bomega)} \ .
\end{align*}
Sending $\alpha \to 0$,
\begin{align} \label{gen el pf 3}
 E(\nu)- E(\mu)  \geq \int \left\la (x_1-x_2)/\tau, x_3-x_2 \right\ra d \bomega + (\lambda/2) \| x_2 - x_3\|_{L^2(\bomega)} \ .
\end{align}
Likewise, we have
\begin{align*}
W_2^2(\nu, \omega) - W_2^2(\mu, \omega) = \int |x_1 - x_3|^2 d \bomega - \int |x_1 - x_2|^2 d \bomega
= \|x_2 - x_3\|_{L^2(\bomega)}^2 + 2 \int \la x_3 - x_2, x_2 - x_1 \ra d \bomega \ .
\end{align*}
Combining this with the fact that $\lambda + 1/\tau >0$,
\begin{align*}
\Phi(\tau, \omega; \nu) - \Phi(\tau, \omega; \mu) \geq \int \la (x_3 - x_1)/\tau +(x_1 - x_3)/\tau, x_2 - x_1 \ra d \bomega =0
\end{align*}
Since $\nu \in D(E)$ was arbitrary, $\mu$ minimizes $\Phi(\tau, \omega; \cdot)$.
\end{proof}

We now turn to the proof of Theorem \ref{proxmapdifftimethm}. For simplicity of notation, we abbreviate $J_\tau \mu$ by $\mu_\tau$.

\begin{thm} \label{proxmapdifftimethm general}
Suppose $E$ satisfies convexity assumption \ref{funas}. Then if $\mu \in \overline{D(E)}$ and $0 < h \leq \tau< \frac{1}{\lambda^-}$,
\[ J_\tau \mu = J_h \left[ \mu^{\mu_\tau \to \mu}_\frac{h}{\tau} \right]  \ , \]
where $\mu^{\mu_\tau \to \mu}_\frac{h}{\tau}$ is any geodesic from $\mu_\tau$ to
$\mu$ at time $\frac{h}{\tau}$.
\end{thm}
\begin{proof}
Choose any geodesic $\mu_\alpha^{\mu_\tau \to \mu}$ from $\mu_\tau$ to $\mu$, and define $\omega = \mu^{\mu_\tau \to \mu}_\frac{h}{\tau}$. We must show $\mu_\tau = \omega_h$.

 By \cite[Lemma 7.2.1]{AGS}, there exists a unique plan $\gamma^{\mu_\tau \to \omega} \in \Gamma_0(\mu_\tau,\omega)$ and there exists $\gamma \in \Gamma_0(\mu_\tau, \mu)$ so that $\gamma^{\mu_\tau \to \omega} = (\frac{\tau -h}{\tau} \pi^{1,1} + \frac{h}{\tau} \pi^{1,2}) \#  \gamma$.
By Theorem \ref{genelthm}, for all $\gamma \in \Gamma_0(\mu_\tau,\mu)$,\begin{align}\label{gen el proof 0}
 E(\nu) - E(\mu_\tau) \geq \int \la x_2, x_3 - x_1 \ra d \bgamma + o\left( \|x_1 - x_3\|_{L^2(\bgamma)} \right)\quad  \forall  \bgamma \in \Gamma(\gamma_\tau, \nu) , \nu \to \mu_\tau \ .
 \end{align}
To prove $ \mu_\tau = \omega_h$, by a second application of Theorem \ref{genelthm}, it's enough to show that for all $\tilde{\gamma} \in \Gamma_0(\mu_\tau,\omega)$,
\begin{align} \label{gen el proof 1}
 E(\tilde{\nu}) - E(\mu_\tau) \geq \int \la x_2, x_3 - x_1 \ra d \tilde{\bgamma} + o\left( \|x_1 - x_3\|_{L^2(\tilde{\bgamma})} \right)\quad  \forall \tilde{\bgamma} \in \Gamma(\tilde{\gamma}_h, \tilde{\nu}), \tilde{\nu} \to \mu_\tau \ .
 \end{align}
Since $\gamma^{\mu_\tau \to \omega}$ is the unique plan in $\Gamma_0(\mu_\tau,\omega)$, it is enough to show that (\ref{gen el proof 1}) holds for $\tilde{\gamma} = \gamma^{\mu_\tau \to \omega}$.

As (\ref{gen el proof 0}) holds for all $\gamma \in \Gamma_0(\mu_\tau, \mu)$, in particular, it holds for $\gamma$ so that $\gamma^{\mu_\tau \to \omega} =(\frac{\tau -h}{\tau} \pi^{1,1} + \frac{h}{\tau} \pi^{1,2}) \#  \gamma$. Furthermore,
\begin{align*}
 \rho_h \left(x_1, \frac{\tau-h}{\tau} x_1 +\frac{h}{\tau} x_2 \right) = \left(x_1, \frac{1}{h} \left[\frac{\tau-h}{\tau} x_1 +\frac{h}{\tau} x_2 - x_1\right] \right)  = \left(x_1, \frac{x_2 - x_1}{\tau} \right) = \rho_\tau(x_1, x_2) \ .
\end{align*}
Consequently, $\tilde{\gamma}_h = \rho_h \# \gamma^{\mu_\tau \to \nu} = \rho_h \circ (\frac{\tau -h}{\tau} \pi^{1,1} + \frac{h}{\tau} \pi^{1,2}) \#  \gamma = \rho_\tau \# \gamma = \gamma_\tau$.
Therefore, the fact that  (\ref{gen el proof 0}) holds for $\gamma$ implies that (\ref{gen el proof 1}) holds for $\tilde{\gamma}$. This proves the result.\end{proof}
}

\begin{acknowledgement}
\noindent \textbf{Acknowledgements:} The author would like to thank Prof. Giuseppe Savar\'e for suggesting this problem. The author would like to thank Prof. Eric Carlen for suggesting the form of Theorem \ref{proxmapdifftimethm} and for many helpful conversations. The author would like to thank Prof. Wilfrid Gangbo for suggestions on the exposition of Section \ref{transport metric section}, particularly regarding the emphasis on the isometry between the transport metric and $L^2$.
\end{acknowledgement}

\bibliographystyle{plain}
\bibliography{ResearchProposalReferences}

\end{document}